\documentclass{article}
\usepackage{graphicx}
\usepackage{tikz}

\usepackage{amsmath}
\usepackage{amssymb}
\usepackage{mathrsfs}
\usepackage[margin=1in]{geometry}
\newcommand{\R}{\mathbb{R}}

\newcommand{\p}{\partial}

\newcommand{\eps}{\varepsilon}
\newcommand{\jbrac}[1]{\langle #1 \rangle}

\newcommand{\les}{\lesssim}
\newcommand{\wkly}{\rightharpoonup}

\newcommand{\Imag}{\operatorname{Im}}

\usepackage{graphicx}
\usepackage{amsthm}
\usepackage{xcolor}
\usepackage{hyperref}

\newtheorem{lem}{Lemma}
\newtheorem{thm}{Theorem}
\newtheorem{prop}{Proposition}
\newtheorem{remark}{Remark}

\newtheorem{defn}{Definition}

\usepackage[backend=biber, style=nature, sorting=nyt]{biblatex}

\title{Modified scattering for semiclassical Bose-Fermi mixtures}
\author{Kieran Cavanagh\footnote{Department of Mathematics, Pennsylvania State University, University Park, PA, USA. Email: kvc5984@psu.edu}}
\date{September 2026}

\begin{document}

\maketitle

\begin{abstract}
    The Vlasov-Hartree system models a mixture of bosons and fermions interacting via Coulomb forces in which the bosons are described quantum-mechanically, while the fermions are described classically. For small initial data, we prove that the system exhibits modified scattering, i.e. the bosonic and fermionic subsystems each converge to solutions of the free Schr\"odinger and Vlasov equations respectively, but with coupled logarithmic corrections owing to the long-range nature of Coulomb interactions. We use a mixed Lagrangian and Fourier-based approach to derive explicit descriptions of the asymptotic dynamics of the solution and associated fields and Lagrangian trajectories. Notably, due to our purely Lagrangian analysis on the Vlasov side, we require no derivatives on the initial fermion density and the asymptotic fermionic profile is as regular as the initial data. Moreover, convergence to the asymptotic bosonic and fermionic profiles occurs in almost the same space as the initial data.
\end{abstract}

\tableofcontents

\section{Introduction}

\subsection{The Vlasov-Hartree system}

In this paper, we study the asymptotic behavior of the Vlasov-Hartree system, which is a mean-field model for a mixture of infinitely many bosons and fermions at zero temperature interacting via a potential $V$, in which the bosons are described quantum-mechanically and the fermions are described classically. The system reads

\begin{equation}\label{VH}
\begin{split}
    \partial_t f + v\cdot\nabla_x f + E\cdot\nabla_v f = 0, \quad E(t,x) = -\nabla_x V*|\phi(t,x)|^2,\\
    i\partial_t \phi + \tfrac12 \Delta_x \phi = (V*\rho)\phi, \quad \rho(t,x) = \int f(t,x,v)dv.
\end{split}
\end{equation}

Here, $(x,v)\in \R^3\times\R^3$ and $f$ is the classical distribution function of the fermions, while $\phi$ is the effective wavefunction for the bosons. The field $E$ generated by the bosons exerts a force on the fermions, while the potential $V*\rho$ generated by the fermions drives the dynamics of the bosons. In this sense, only fermion-boson interactions are considered, as fermion-fermion and boson-boson interactions as modeled by the Vlasov-Poisson or Hartree equations respectively have already been considered extensively in the literature. Moreover, we consider long-range interactions, given by the Coulomb potential $$V(x) = \frac{\lambda}{4\pi|x|}, \quad \lambda \in \{-1,1\}.$$
When $\lambda=1$, the interaction is repulsive, and when $\lambda=-1$ the interaction is attractive. Because we will work in a perturbative regime where conservation laws play a minimal role, we set $\lambda=1$ from here on. All results will hold in an analogous fashion when $\lambda=-1$.\\

Recently, C\'ardenas, Miller, and Pavlovi\'c \cite{cardenas_effective_2025} derived the Vlasov-Hartree system from the finite particle system in a particular nontrivial scaling regime where the bosons are much lighter than the fermions, but are more numerous. To be precise, if $N$ is the number of bosons, $M$ is the number of fermions, and $m_B,m_F$ are the masses of the bosons and fermions respectively, they prove that under the scaling
$$\frac{M}{N} = \frac{m_B}{m_F} = \hbar,$$
the Vlasov-Hartree system arises from the $N,M\to \infty$, $\hbar\to 0$ limit. For more physical background on the Vlasov-Hartree system, we refer to \cite{cardenas_effective_2025} and our earlier work \cite{cavanagh_global_2026}.\\

From a mathematical perspective, this system is an example of a kinetic transport equation coupled to a dispersive equation, which appear in many other contexts, including relativistic plasmas (e.g. the Vlasov-Maxwell system) and general relativity (e.g. the Vlasov-Einstein system). Understanding the asymptotic dynamics of these equations is often a difficult task requiring delicate analysis. One might expect solutions to scatter, i.e. for the asymptotic behavior to be well-approximated as $t\to\infty$ by \emph{free solutions} (solutions in which the particles experience no force) but for long-range potentials, such as the Coulomb potential that we consider in this paper, the particles do not disperse fast enough, causing deviation from the free dynamics. Nonetheless, in many cases, including for the present system, one can establish convergence to an asymptotic dynamics modulo a correction, which is known as \emph{modified scattering}.\\

The Cauchy problem for the Vlasov-Hartree system has only very recently received attention in \cite{cavanagh_global_2026} in which global well-posedness and time decay estimates are established for large data and the very recent preprint \cite{huang_long-time_2026} in which similar results to the present paper are obtained using different methods and under stronger regularity hypotheses on the initial data. For the purposes of motivating the main results we will mainly discuss the closely related Vlasov-Poisson and Hartree equations, since their analysis forms the basis of the methods of the present paper. The former equation describes either a self-gravitating system of particles governed by Newton's laws or an electrostatic plasma, while the latter describes a Bose-Einstein condensate.\\

\textbf{a. The Hartree equation.} The global well-posedness theory was first developed by Chadam and Glassey \cite{chadam_global_1975}, and the impossibility of linear scattering was established by Glassey in \cite{glassey_asymptotic_1977}. Hayashi and Tsutsumi \cite{hayashi_scattering_1987} established the linear scattering theory for short-range interactions. Later, for Coulomb interactions, Ginibre and Ozawa \cite{ginibre_long_1993} first constructed modified wave operators --- operators which map (modified) asymptotic states to initial data --- while Hayashi and Naumkin \cite{hayashi_asymptotics_1998} were the first to prove the sharp $L^\infty$ decay estimate for solutions alongside modified scattering. The works of Ginibre and Velo \cite{ginibre_long_2000} and Nakanishi \cite{nakanishi_modified_2002} further developed the scattering theory to include longer-range potentials and to cover a larger class of scattering states. Later, Kato and Pusateri \cite{kato_new_2011} gave a new proof of the results of \cite{hayashi_asymptotics_1998} inspired by the spacetime resonance method. \\

\textbf{b. The Vlasov-Poisson system.} Global well-posedness for large initial data was established by Pfaffelmoser \cite{pfaffelmoser_global_1992} and Perthame and Lions \cite{lions_propagation_1991}, while small data global well-posedness with sharp decay estimates was established by Bardos and Degond \cite{bardos_global_1985}, with extensions in \cite{hwang_optimal_2011}, \cite{smulevici_small_2016}, \cite{duan_sharp_2023}, \cite{wang_decay_2023}. The modified scattering theory for the Vlasov-Poisson system came much later from Choi and Kwon \cite{choi_modified_2016}, and later works from Ionescu, Pausader, Widmayer, and Wang \cite{ionescu_asymptotic_2022} and Pankavich \cite{pankavich_asymptotic_2022} derived more explicit descriptions of the asymptotic dynamics. Most of the above results apply only in the case of small initial data, except \cite{pankavich_asymptotic_2022} in which an a priori decay assumption on the electric field $E$ (without smallness) is sufficient to obtain modified scattering. This assumption is known to hold for small solutions or for spherically symmetric solutions \cite{pankavich_exact_2021}. The past several years have seen a flurry of work on the asymptotic behavior of the Vlasov-Poisson system in $\R^3$, for instance see \cite{flynn_scattering_2023}, \cite{schlue_inverse_2025}, \cite{bigorgne_homeomorphic_2026} for the construction of modified wave operators, \cite{pausader_stability_2021}, 
\cite{bigorgne_modified_2025}, \cite{kepka_modified_2025}, \cite{huang_scattering_2026}, \cite{pausader_stability_2026} for the study of asymptotic behavior under the influence of a point charge or external potential, \cite{huang_scattering_2026-1}, \cite{hong_modified_2026} for very long-range interactions, and \cite{ionescu_stability_2023}, \cite{ionescu_nonlinear_2024} for stability of spatially homogeneous steady states in $\R^3$.

\subsection{Main results} 

The goal of this work is to study the asymptotic behavior of the Vlasov-Hartree system in the small initial data regime. Importantly, we consider data with possibly \emph{discontinuous} initial fermion density $f_0$. Ground states (specifically, steady states which minimize energy under the Pauli exclusion principle $0\leq f\leq 1$) for the fermionic component of the Vlasov-Hartree system are in general indicator functions due to the fact that the system is at zero temperature; see for instance \cite{fournais_semi-classical_2018} or \cite{cardenas_effective_2025}. As a result, a low-regularity setting is the most physically natural one and enables future study of the stability of such ground states.\\

Our first result concerns optimal decay estimates for initial data small in essentially the same class of data for which we showed global well-posedness in \cite{cavanagh_global_2026}. 

\begin{thm}[Decay estimates]\label{decay}
    Suppose $f_0\in L^\infty_{x,v}$ with compact support and $\phi_0\in H^s, |x|^s\phi_0\in L^2$ for some $s\in (\frac32,2)$. There exists $\eps_0$ sufficiently small such that if
    $$\|f_0\|_{L^1_{x,v}\cap L_{x,v}^\infty}  + \|\phi_0\|_{H^s_x} + \||x|^s\phi_0\|_{L^2_x} \leq \eps_0,$$
    then the densities decay at the same rate as free solutions, i.e.
    $$\|\rho(t)\|_{L^\infty}\les \eps_0 \jbrac{t}^{-3}, \quad \|\phi(t)\|_{L^\infty}\les \eps_0 \jbrac{t}^{-3/2},$$
    which in particular implies the following decay rates on the fields:
    \begin{align*}
        \||\nabla|^r V*\rho(t)\|_{L^\infty} &\les \eps_0 \jbrac{t}^{-1-r} \quad  \forall r\in [0,2),\\
        \|E(t)\|_{L^\infty} &\les \eps_0^2 \jbrac{t}^{-2}, \quad \|\nabla_x E(t)\|_{L^\infty} \les \eps_0^2 \jbrac{t}^{-3}\ln\jbrac{t}.
    \end{align*}
\end{thm}

Our next result concerns the asymptotic behavior. We show that both the bosonic and fermionic subsystems exhibit modified scattering, and we give explicit descriptions of the asymptotic dynamics and long-range corrections. In order to state the results we need some notation. First, define
\begin{equation}\label{sgap}
    s_{gap} = s-\frac32 > 0.
\end{equation}

The quantity $s_{gap}$ measures the amount that $s$ is above the Sobolev embedding threshold and also is the H\"older regularity of the data for $\phi$ by Morrey's inequality.\\

We also define the \emph{scattering mass} or \emph{spatial average} $m(t,v)$ (see \cite{ionescu_asymptotic_2022}, \cite{pankavich_asymptotic_2022}), which will play an important role in the description of the asymptotic dynamics, as it determines how much average force is exerted on a boson with velocity $v$.

$$m(t,v) := \int f(t,x,v)dx.$$

We introduce the Lagrangian flow map 
$$\Phi_{t,s}(x,v):=(X(t,s,x,v),V(t,s,x,v)),$$
where $X,V$ solve the characteristic ODEs (\ref{charodes}), and let $\mathcal{S}_t\subset \R^6$ be the support of $f(t)$.

\begin{thm}[Modified scattering]\label{modscat}

Let $\delta<\min(\frac13 s_{gap},\frac18)$. Under the hypotheses of Theorem \ref{decay},\\ 

    (i) (Hartree asymptotics) There exists an asymptotic profile $\phi_\infty\in L^2\cap L^\infty \cap C^\delta$, a real phase function $\Psi_\infty\in L^\infty$, and $m_\infty \in L^\infty$ with compact support such that
    \begin{equation}\label{phiasymp}
        \Big\|\phi(t,x) - \frac{e^{i|x|^2/2t}}{(it)^{3/2}} \phi_\infty\big(\tfrac{x}{t}\big) \exp\left(-i\Psi_\infty\big(\tfrac{x}{t}\big) - i\ln(t)(V*m_\infty)\big(\tfrac{x}{t}\big) \right)\Big\|_{L^p_x} \les \eps_0 t^{-\delta - 3(\frac12-\frac1p)},
    \end{equation}
    for any $2\leq p\leq \infty$. Also, the potential $V*\rho$ driving the boson dynamics has the self-similar expansion
    \begin{equation}\label{VrhoSSasymp}
        \big\|(V*\rho)(t,x) - t^{-1}(V*m_\infty)\big(\tfrac{x}{t}\big)\big\|_{L^\infty_x} \les \eps_0 t^{-2}\ln^5\jbrac{t},
    \end{equation}
    and $m_\infty$ arises as the weak limit in the sense of measures of $m(t)$:
    \begin{equation}\label{masymp}
        m(t,v)dv\wkly m_\infty(v)dv \quad \text{ as } t\to\infty.
    \end{equation}

    (ii) (Vlasov asymptotics) Let $E_\infty(v)=-(\nabla V*|\phi_\infty|^2)(v)$ be the asymptotic field profile generated by the bosons and $$\mathcal{M}_t(x,v)=(x+vt-\ln(t)E_\infty(v),v)$$
    the modified free flow. Then, for any bounded open set $U\subset\R^6$ and $s\geq 1$ there exists an asymptotic flow map $\mathcal{A}_{\infty,s}:U\to\mathcal{A}_{\infty,s}(U)$ which is a $C^1_{x,v}$ measure-preserving diffeomorphism such that
    \begin{gather}\label{Aasymp}
        \mathcal{M}_t^{-1}\circ\Phi_{t,s} \to \mathcal{A}_{\infty,s} \text{ locally in } C^1_{x,v} \text{ as } t\to\infty,\\
        \mathcal{A}_{\infty,s} = \mathcal{M}_s^{-1} + O(s^\delta) \text{ on } \mathcal{S}_s \text{ as } s\to\infty
    \end{gather}
    and an asymptotic profile $f_\infty\in L^\infty_{x,v}$ with compact support defined by
    $$f_\infty = (f(s)\circ \mathcal{A}_{\infty,s}^{-1})\chi_{\mathcal{A}_{\infty,s}(\mathcal{S}_s)} $$
    such that
    \begin{equation}\label{fasymp}
        \lim_{t\to\infty}\|f(t)\circ \mathcal{M}_t-f_\infty\|_{L^p_{x,v}}=0, \quad \forall 1\leq p < \infty.
    \end{equation}
    Moreover, the force acting on the fermions satisfies
    \begin{equation}\label{Easymp}
        \|E(t,x) - t^{-2}E_\infty\big(\tfrac{x}{t}\big)\|_{L^\infty_x} \les \eps_0^2 t^{-2-\delta},
    \end{equation}
    and the spatial fermion density satisfies
    \begin{equation}\label{rhoasymp}
        t^3 \rho(t,tv)dv \wkly m_\infty(v)dv \quad \text{ as } t\to\infty,
    \end{equation}
    weakly in the sense of measures.
\end{thm}

\begin{remark}
     Smallness of the initial data does not seem to be a structural requirement for Theorem \ref{modscat} to hold; one could assume a priori decay of $E$ and derivatives without smallness assumptions and still obtain modified scattering by using the techniques of Pankavich in \cite{pankavich_asymptotic_2022} to analyze the Vlasov flow, with no modifications to the analysis on the Hartree side.
\end{remark}
\begin{remark}
    Technically, the only statement requiring $\delta<\min(\frac13 s_{gap},\frac18)$ is $\phi_\infty\in C^\delta$ (see Lemma \ref{remest}); the other statements hold as long as $\delta<\frac12 s_{gap}$.
\end{remark}

\begin{remark}
    If one assumes $f$ is more regular, one can use the techniques of, for instance, \cite{pankavich_asymptotic_2022} and \cite{ionescu_asymptotic_2022}, to prove convergence to the asymptotic Vlasov profile in $L^\infty$ with quantitative rates, and similarly the weak convergence statements above can be upgraded to strong convergence with rates.
\end{remark}



\subsection{Difficulties and sketch of the proof}

\textbf{a. The Hartree part.} The key difficulty in the analysis of the Hartree part is obtaining the decay estimate $\|\phi(t)\|_{L^\infty}\les t^{-3/2}$. Once this is established, the required decay estimates for $E(t)$ follow from standard potential estimates and one can prove the necessary decay of $\rho(t)$ and $V*\rho(t)$ by showing, as in Bardos-Degond \cite{bardos_global_1985}, that the Vlasov characteristics behave like those of free transport plus lower order terms, see Lemma \ref{geoflow} and Lemma \ref{rhodecay}. So, the proof of Theorem \ref{decay} may be reduced to the $L^\infty$ decay estimate for $\phi$.\\

As is expected for long-range problems, the nonlinearity $(V*\rho)\phi(t)$ decays too slowly to treat it as purely perturbative. We draw inspiration from the analysis of the usual Hartree equation (see \cite{hayashi_asymptotics_1998}, \cite{kato_new_2011}) to get around this issue. To factor out the effects of the free dynamics, introduce the \emph{profiles} $$\psi(t) := e^{-it\Delta/2}\phi(t), \quad g(t,x,v) := f(t,x+vt,v).$$

The Fraunhofer formula (Lemma \ref{fraunhofer})

$$\phi(t,x) = \frac{e^{i|x|^2/2t}}{(it)^{3/2}}\widehat\psi(t,\tfrac{x}{t}) + o(t^{-3/2}),$$

which can be thought of as a stationary phase refinement of the usual $L^1-L^\infty$ dispersive estimate for $e^{it\Delta/2}$, reduces the $L^\infty$ decay estimate for $\phi$ to a uniform bound on $\|\hat\psi(t)\|_{L^\infty_{\xi}}$, since the remainder is controlled by weighted $L^2$ norms of $\psi$. We control these weighted norms using energy estimates with the Galilean vector field $J:=x+it\nabla_x$, which commutes with the linear Schr\"odinger equation and captures the $L^2$ evolution of moments along the linear flow, i.e. $\|J^s\phi\|_{L^2}=\||x|^s\psi\|_{L^2}$, see Proposition \ref{EEhartree} for details.\\

This brings us to the central difficulty: obtaining a uniform bound on $\|\hat\psi(t)\|_{L^\infty_\xi}$. Following the work of Hayashi-Naumkin \cite{hayashi_asymptotics_1998} and Kato-Pusateri \cite{kato_new_2011}, one needs to extract a phase correction, which requires determining the leading dynamics of the evolution equation for $\hat\psi$. To do so, we perform an explicit computation in Fourier space akin to \cite{kato_new_2011} or the proof of Fraunhofer's formula, making crucial use of the transport structure of the Vlasov equation, to find
\begin{equation}\label{hartreeleading}
    \p_t \hat\psi(t,\xi) = -it^{-1}(V*m)(t,\xi) \hat\psi(t,\xi) + O(t^{-1-\delta}).
\end{equation}
Then, provided one can control the remainder (which is delicate and requires bounds on $L^p$ moments of both $\psi$ and $g$, see Lemma \ref{remest}), an integrating factor can be used to remove the first term, so that the modified profile $$\hat{w}(t,\xi):= e^{iP(t,\xi)}\hat\psi(t,\xi), \quad P(t,\xi):=\int_1^t s^{-1}(V*m)(s,\xi)ds,$$ satisfies $\p_t \hat{w}(t,\xi) = O(t^{-1-\delta})$ and so one can integrate to obtain a uniform bound on $\|\hat{w}(t)\|_{L^\infty_\xi}$ and therefore $\|\hat{\psi}(t)\|_{L^\infty_\xi}$ since $P$ is real.\\

As a result, one also obtains $L^2_\xi\cap L^\infty_\xi$ convergence of $\hat{w}(t)$ to an asymptotic profile. In fact, for later purposes we show convergence in $C^\delta_\xi$ where $\delta$ satisfies the conditions of Theorem \ref{modscat}. This, when combined with Fraunhofer's formula, is almost enough to establish the asymptotic formula (\ref{phiasymp}), but we still need to show that the phase $P(t,\xi)$ can be well-approximated by $\ln(t)(V*m_\infty)(\xi)$. This follows from the convergence of the velocity characteristics on the Vlasov side combined with some delicate potential estimates, see Proposition \ref{scatmass}.\\

\textbf{b. The Vlasov part.} The next key difficulty is proving convergence to an asymptotic Vlasov profile $f_\infty$ without assuming any differentiability on the initial data $f_0$. We cannot use the Vlasov equation since that would involve controlling derivatives, so we prove convergence of the (suitably renormalized) particle trajectories instead. The convergence of the velocity characteristics is straightforward and follows from the decay estimate $\|E(t)\|_{L^\infty}\les t^{-2}$. However, since this decay estimate is critical from the standpoint of the spatial characteristics---indeed, integrating $E(t) = O(t^{-2})$ twice as Newton's laws would suggest leads to $O(\ln t)$ growth of the spatial characteristics---one needs to extract the leading dynamics of $E$ in order to prove convergence of the spatial characteristics. We use a similar Fourier-based calculation to extract the leading dynamics of $E$ along the free Vlasov flow:
$$E(t,x+vt) = -t^{-2}[\nabla V * |\hat\psi(t)|^2](v) + o(t^{-2})$$

Since $|\hat\psi(t)|^2$ converges to the asymptotic profile $|\phi_\infty|^2$, this leads to the definition of $E_\infty(v)$ in the statement of Theorem \ref{modscat}. The error is controlled by weighted norms of the Vlasov profile. Then, since $E(t,x+vt)\approx t^{-2}E_\infty(v)$, Newton's laws suggest that the fermion's phase-space evolution should behave like $$\mathcal{M}_t(x,v) := (x+vt-\ln(t)E_\infty(v),v),$$ which we call the \emph{modified free flow}. In this sense, we prove convergence of the particle trajectories \emph{in the frame of the modified free flow}, i.e. we show that $\mathcal{M}_t^{-1}\circ\Phi_{t,s}$ has a limit. An analogous argument has been carried out for Vlasov-Poisson in Pankavich \cite{pankavich_asymptotic_2022}.\\

However, one encounters several difficulties due to the low-regularity setting. First, the arguments of \cite{pankavich_asymptotic_2022} require $E_\infty\in C^2$ which is far more regularity than we have here, since $\phi_\infty\in C^\delta$ and elliptic regularity gives $E_\infty\in C^{1,\delta}$. This is because $\mathcal{M}_t^{-1}\circ\Phi_{t,s}$ requires evaluating $E_\infty$ along the velocity characteristics and hence taking a time derivative and a gradient to establish $C^1$ convergence puts two derivatives on $E_\infty$. To get around this, we ``freeze" the velocity at infinity and use the slight amount of H\"older regularity of $\nabla E_\infty$ to estimate the deviation between $\nabla E_\infty(V(t))$ and $\nabla E_\infty(V_\infty)$; see Proposition \ref{Xlim}. The second key difficulty is that one needs to show that $f(t)\circ \mathcal{M}_t = f(s)\circ (\mathcal{M}_t^{-1}\circ\Phi_{t,s})^{-1}$ converges to a limit in $L^p$ as $t\to\infty$, but the convergence of the renormalized flow may not behave well with respect to precomposition by the possibly discontinuous function $f(s)$. The measure-preserving structure of the asymptotic flow is essential here and allows us to use a density argument to upgrade the weak convergence in the sense of measures that one gets fairly easily to strong $L^2$ convergence, from which $L^p$, $1\leq p<\infty$ convergence follows by interpolation.

\subsection{Key contributions}

\textbf{a. The low-regularity Lagrangian setting.} Notably, we require no derivatives on the initial Vlasov data $f_0$, and essentially the minimum regularity on the initial Hartree data $\phi_0$. In fact, from the standpoint of regularity, we prove modified scattering in essentially the same class of data that we proved local and global well-posedness in our earlier work \cite{cavanagh_global_2026}. It is unlikely one can push the required regularity on $\phi_0$ down further without risking non-uniqueness, since $\phi(t)\in H^s$ for $s<3/2$ is not enough to ensure $E$ is Lipschitz and hence produce a unique Lagrangian flow. With the exception of \cite{kepka_modified_2025}, every existing modified scattering result for Vlasov-Poisson requires at least one derivative on the initial Vlasov data.\\

Because we require no derivatives on $f$, we are forced to use a purely Lagrangian framework for the Vlasov analysis and cannot use tools like vector fields or PDE methods more broadly. However, our Lagrangian framework has the benefit of respecting the natural structure of the Vlasov equation. While classical, we believe this approach is still very fruitful, especially for low-regularity problems like the present one. Recent works \cite{flynn_scattering_2023}, \cite{pausader_stability_2021}, \cite{kepka_modified_2025}, \cite{pausader_stability_2026}, similarly leverage the Hamiltonian structure of the Vlasov-Poisson system to study its asymptotic behavior, illustrating the utility of such an approach. Our low-regularity framework can hopefully be used to tackle other low-regularity problems, such as the stability of (discontinuous) ground states for the Vlasov-Hartree system, or the problem of modified scattering for Vlasov-Poisson system with minimal derivative assumptions on the initial data, which was listed as a problem of interest in \cite{pausader_stability_2026-1}.\\

\textbf{b. Detailed asymptotics.} Another benefit of our approach is that we can relatively easily obtain detailed asymptotics of most quantities associated to the solution since most quantities can be written in terms of the characteristics, which converge. Quantities which enjoy elliptic regularization have asymptotics in stronger topologies with explicit rates (e.g. (\ref{VrhoSSasymp}), (\ref{Easymp})), while quantities which do not enjoy elliptic regularization generally have asymptotics in weaker topologies (e.g. (\ref{masymp}), (\ref{rhoasymp})).

\subsection{Future questions}

\textbf{a. Inverse scattering.} To complete the analysis of the asymptotic behavior, it would be desirable to construct (modified) wave operators which map asymptotic states to initial data, which would help answer the question of which final states are ``reachable" by the nonlinear dynamics. This has been accomplished for the Vlasov-Poisson and Hartree equations separately, see for instance \cite{ginibre_long_1993}, \cite{ginibre_long_2000}, \cite{nakanishi_modified_2002} for Hartree and \cite{flynn_scattering_2023}, \cite{schlue_inverse_2025}, \cite{bigorgne_homeomorphic_2026} for Vlasov-Poisson.\\

\textbf{b. Stability of other equilibria.} It would be interesting to consider the stability and dynamical behavior of the Vlasov-Hartree system near nonzero equilibrium solutions. For spatially homogeneous equilibria to the Vlasov-Poisson system, this falls in the realm of Landau damping \cite{ionescu_stability_2023} \cite{ionescu_nonlinear_2024}, but much less is known for non-homogeneous equilibria. For the Vlasov-Hartree system, the fermionic part of the ground state is generally an indicator function; see \cite{fournais_semi-classical_2018}, \cite{cardenas_effective_2025}. As previously mentioned, this is one motivation for our study of low-regularity solutions, since in general such an equilibrium would be discontinuous. We also note that our previous work \cite{cavanagh_global_2026} shows that nontrivial equilibria with finite mass, energy, and spatial moments only exist in the case of attractive interactions ($\lambda=-1$).\\

\textbf{c. Relativistic variants.} It would be interesting to consider a (semi)relativistic variant of the present system, in which $v$ in the Vlasov equation is replaced by $\hat{v}:=
v (1+|v|^2)^{-1/2}$ and the Schr\"odinger operator is replaced by the half Klein-Gordon operator $\sqrt{-\Delta+1}$. Modified scattering for the semirelativistic Hartree equation, also called the boson star equation, was carried out in \cite{pusateri_modified_2014}. We remark that a different coupling of a Klein-Gordon equation to a relativistic transport equation has been studied in \cite{nguyen_new_2024}, in which a Lagrangian approach was used to analyze the interaction between the Klein-Gordon waves and Vlasov particles.

\subsection{Plan of the paper}

In Sections \ref{prelim} through \ref{boson} we develop the necessary tools to prove Theorem \ref{modscat} and along the way prove Theorem \ref{decay}. In particular, the preliminaries including the bootstrap argument used to prove Theorem \ref{decay} are covered in Section \ref{prelim}, while the estimates pertaining to the fermions and bosons are proved in Sections \ref{ferm} and \ref{boson} respectively. Then, in Section \ref{asympbehavior} we complete the proof of Theorem \ref{modscat} and also prove Corollary \ref{higherreg}.\\

\textbf{Acknowledgments.} I would like to thank my advisor, Anna Mazzucato, for support and encouragement. AI tools were used to check the manuscript for errors, but the text of the manuscript is 100\% human-written and all mathematical ideas and proofs were developed by the author without the assistance of AI. 

\section{Preliminaries}\label{prelim}

\subsection{Notation}
\begin{itemize}
    \item We will use the Fourier transform and its inverse frequently. 

$$\hat{f}(\xi) = (\mathcal{F}f)(\xi) := \int e^{-ix\cdot\xi} f(x)dx, \quad (\mathcal{F}^{-1}f)(x) := (2\pi)^{-3}\int e^{ix\cdot\xi}f(\xi)d\xi.$$

\item For $x\in \R^n$ we denote $\jbrac{x} = \sqrt{2+|x|^2}$; note that $\ln\jbrac{0} \neq 0$ with this notation.

\item For the Lebesgue spaces $L^p$, we will often write $L^p_x$ if $x$ is the integration variable, and similarly for other variables.
\item In the context of regularity indices, $s$ will always refer to the $L^2$-based Sobolev regularity index of $\phi_0$, although we also use $s$ as a time variable. In context, it will be clear whether we mean $s$ as a regularity index or as a time variable.
\end{itemize}

\subsection{Potential estimates}

The following lemma will be used often to control the expressions involving the Coulomb potential.

\begin{lem}\label{potest}
    Let $0<\alpha<3$.\\
    
    (i) Denote $r_\alpha = \frac{3}{3-\alpha}$. Then as long as $1\leq p<r_\alpha<q\leq \infty$, 
    \begin{equation*}
        \||x|^{-\alpha}*u\|_{L^\infty} \les \|u\|_{L^p}^{\theta}\|u\|_{L^q}^{1-\theta}, \quad \theta = \frac{\frac{1}{r_\alpha}-\frac1q}{\frac1p - \frac1q}.
    \end{equation*}
    
    (ii) For any $ 3< p\leq \infty$ and $0<r<R$, we have
    \begin{equation*}
        \|\nabla_x^2 (V*u)\|_{L^\infty} \les r^{1-3/p}\|\nabla u\|_{L^p} + (1+\ln(R/r))\|u\|_{L^\infty} + R^{-3}\|u\|_{L^1}.
    \end{equation*}
\end{lem}

\begin{proof} The lemma is well-known; for a complete proof see for instance our earlier work \cite{cavanagh_global_2026}.
\end{proof}

\subsection{Bootstrap argument}

First, recall the notion of a Lagrangian weak-mild solution to (\ref{VH}) and the global existence theorem from our previous work \cite{cavanagh_global_2026}:

\begin{defn}
    We say $(f,\phi)$ is a Lagrangian weak-mild solution to (\ref{VH}) on $[0,T]$ if 
    \begin{itemize}
        \item $(f,\phi)$ solves (\ref{VH}) in the sense of distributions;
        \item $(f,\phi)\in L^\infty([0,T];L^1_{x,v}\cap L^\infty_{x,v})\times C([0,T];H^s_x)$ for some $s>3/2$ and $f(t)$ is compactly supported for every $t\in [0,T]$;
        \item $\phi$ is a mild solution to the Hartree equation in (\ref{VH}), i.e. 
        $$\phi(t) = e^{it\Delta/2}\phi(0) - i\int_0^t e^{i(t-s)\Delta/2}[(V*\rho)\phi](s)ds;$$
        \item Energy is conserved for almost every $t\in [0,T]$, i.e.
        $$\mathcal{E}(t) := \frac12 \int |\nabla_x\phi(t)|^2 dx + \frac12 \iint |v|^2 f(t)dxdv + \int (V*\rho(t))|\phi(t)|^2 dx = \mathcal{E}(0).$$
     \end{itemize}
\end{defn}

\begin{thm}[\cite{cavanagh_global_2026}]\label{globalex}
    Let $f_0\in L^\infty_{x,v}$ with compact support and $\phi_0\in H^s_x$ for some $s\in (3/2,2)$, and assume $\lambda\in \{-1,1\}$. Then for any $T<\infty$ there exists a unique Lagrangian weak-mild solution to (\ref{VH}) on $[0,T]$ with initial data $(f_0,\phi_0)$.
\end{thm}

\begin{remark}
Recall that the regularity assumption on $\phi$ implies that $E$ is Lipschitz (actually $C^{1,\alpha}$ for $\alpha = s-\frac32$ by elliptic regularity and Morrey's embedding) and so the Lagrangian flow is well-defined.
\end{remark}

The proof of Theorem \ref{decay} will follow from the following bootstrap argument. We assume the field $E$ decays at the rate of a free solution, via the assumption that 

$$z(t) := \sup_{\tau\in[0,t]}\Big( \jbrac{\tau}^2\|E(\tau)\|_{L^\infty}+\jbrac{\tau}^3\ln^{-1}\jbrac{\tau}\|\nabla_x E(\tau)\|_{L^\infty}\Big)\leq \eps$$

up to time $T$. Then, provided $\eps$ is small, the decay estimates for $E$ will be used to prove that the nonlinear Lagrangian flow closely approximates the flow associated to free transport, which in particular implies that the fermion density $\rho$ and the associated potential $V*\rho$ decay at the same rate as a free solution. These estimates give bounds on spatial moments of the profiles $g,\psi$, which allow us to construct a phase correction and thereby prove the decay estimate for $\phi$ as outlined in the introduction. This will imply the following improved bootstrap estimate for $E$, namely
$$z(t) \les \eps_0^2(1+\eps)^2,$$
which will close the bootstrap upon taking $\eps_0$ (the size of the initial data) sufficiently small. The sketch of the argument is outlined in Figure \ref{fig:bootstrapfig}.

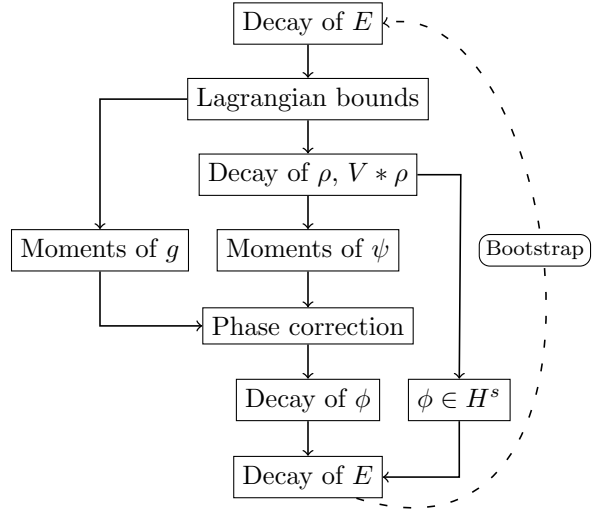
\begin{figure}[h]
    \centering
    \caption{Sketch of the bootstrap argument used to obtain the decay estimates}
    \begin{tikzpicture}
  \node[draw] (node3) at (2,3) {Decay of $E$};
  \node[draw] (node4) at (2,2) {Lagrangian bounds};
  \node[draw] (node6) at (-0.75,0) {Moments of $g$};
  \node[draw] (node5) at (2,1) {Decay of $\rho$, $V*\rho$};
  \node[draw] (node1) at (2,0) {Moments of $\psi$};
  \node[draw] (node7) at (4,-1.98) {$\phi\in H^s$};
  \node[draw] (node2) at (2,-1.98) {Decay of $\phi$};
  \node[draw] (node8) at (2,-1) {Phase correction};
  \draw[->, semithick] (node3.south) -- (node4.north);
  \draw[->, semithick] (node4.south) -- (node5.north);
  \draw[->, semithick] (node5.south) -- (node1.north);

  \node[draw] (node9) at (2,-3) {Decay of $E$};

  \draw[->, semithick] (node5.east) -- (4.02,1) -- (node7.north);

  \draw[->, semithick] (node4.west) -- (-0.75,2) -- (node6.north);

  \draw[->, semithick] (node6.south) -- (-0.75,-1) -- (node8.west);
  \draw[->, semithick] (node1.south) -- (node8.north);

  \draw[semithick, ->] (node7.south) -- (4,-3) -- (node9.east);
  \draw[->, semithick] (node2.south) -- (node9.north);
  \draw[->, semithick] (node8.south) -- (node2.north);
  \draw[->, semithick, loosely dashed] (2.64,-3.28) .. controls (6.5,-4.5) and (5.32,3.5) .. (node3.east);
  \node[draw=black, fill=white, inner sep=2.43pt, node font=\footnotesize, rounded corners] at (5,0) {Bootstrap};
\end{tikzpicture}
    \label{fig:bootstrapfig}
\end{figure}

We make the above discussion rigorous in the following proposition, whose proof is the main content of the next few sections.

\begin{prop}\label{bootstrap}
    Let $T>0$ be arbitrary. Suppose $z(t)\leq \eps$ for all $t\in [0,T]$. Then there exists $\eps_*>0$ such that if $\eps\leq \eps_*$, the following decay estimates hold for all $t\in[0,T]$:
    \begin{gather}
        \|\rho(t)\|_{L^\infty}\les \eps_0\jbrac{t}^{-3}, \quad \||\nabla|^r V*\rho(t)\|_{L^\infty}\les \eps_0\jbrac{t}^{-1-r}, \label{rhodecaybs}\\
        \|\phi(t)\|_{L^\infty}\les \eps_0(1+\eps)\jbrac{t}^{-3/2}.\label{phidecaybs}
    \end{gather}
    And as a consequence,
    \begin{equation}\label{improvedbs}
        z(t)\leq C\eps_0^2(1+\eps)^2 \quad \forall t\in[0,T].
    \end{equation}
\end{prop}

\begin{proof}[Proof of Theorem \ref{decay} assuming Proposition \ref{bootstrap}]
    We need to first verify that $z(t)$ is continuous. From the existence theorem in \cite{cavanagh_global_2026}, $\phi\in C([0,T];H^s)$ and therefore by Sobolev embedding $\phi\in C([0,T];L^2\cap L^\infty)$. Moreover, using Lemma \ref{potest}, we get
    \begin{align*}
        \|E(t)-E(t')\|_{L^\infty} &\les \big\||\phi(t)|^2-|\phi(t')|^2 \big\|_{L^1}^{1/3} \big\||\phi(t)|^2-|\phi(t')|^2 \big\|_{L^\infty}^{2/3}\\
        &\les \|\phi\|_{C([0,T];L^2\cap L^\infty)} \|\phi(t)-\phi(t')\|_{L^2}^{1/3}\|\phi(t)-\phi(t')\|_{L^\infty}^{2/3}.
    \end{align*}
    Therefore, $E\in C([0,T];L^\infty)$. Similarly, using the second claim in Lemma \ref{potest} with $r=1$, $R=2$, and $p=\frac{6}{3-2(s-1)}=\frac{6}{5-2s}$ (which is larger than $3$ since $s>3/2$), so that $\dot{H}^{s-1}\hookrightarrow L^p$, we get
    \begin{align*}
        \|\nabla_x E(t)-\nabla_x E(t')\|_{L^\infty} &\les \|\nabla_x(|\phi(t)|^2-|\phi(t')|^2)\|_{L^p} + \||\phi(t)|^2-|\phi(t')|^2\|_{L^\infty} + \||\phi(t)|^2-|\phi(t')|^2\|_{L^1}\\
        &\les \||\phi(t)|^2-|\phi(t')|^2\|_{H^s}\\
        &\les \|\phi\|_{C([0,T];H^s)}\|\phi(t)-\phi(t')\|_{H^s}.
    \end{align*}
    This proves that $\nabla_x E\in C([0,T];L^\infty)$.\\
    
    Now we prove that the bootstrap proposition implies the theorem. Without loss of generality, assume $\eps_*\leq 1$. We will prove that if $\eps_0$ is chosen small enough, then $z(t)\leq \eps_*$ on $[0,T]$. To this end, suppose there exists a finite positive time $T_*:=\inf\{t\in [0,T] : z(t)>\eps_*\}$, so that $z(t)\leq \eps_*$ whenever $t\in [0,T_*)$. By continuity, $z(T_*)\leq \eps_*$ as well, so (\ref{improvedbs}) gives $z(t)\leq C\eps_0^2(1+\eps_*)^2$ on $[0,T_*]$. By choosing $\eps_0^2\leq \frac{\eps_*}{2C(1+\eps_*)^2}$, we get $z(t)\leq \eps_*/2$ on $[0,T_*]$, contradicting the definition of $T_*$ as we can extend the interval on which $z(t)\leq \eps_*$ by continuity. As a result, the decay estimates stated in Theorem \ref{decay} follow from (\ref{rhodecaybs}),(\ref{phidecaybs}) after applying Proposition \ref{bootstrap} again.
\end{proof}

\section{Analysis of the fermionic part}\label{ferm}

\subsection{Preliminary estimates on the characteristic flow}
For any $s,t\geq 0$, define the characteristic ODEs
\begin{equation}\label{charodes}
    \begin{cases}
        \frac{d}{ds} X(s,t,x,v) = V(s,t,x,v) & X(t,t,x,v) = x,\\
        \frac{d}{ds}V(s,t,x,v) = E(s,X(s,t,x,v)) & V(t,t,x,v)=v.
    \end{cases}
\end{equation}
When $s>t$ we call $(X,V)$ \emph{forward} characteristics, and when $s<t$ we call them \emph{backward} characteristics. In any case, they generate the nonlinear flow 
$$\Phi_{s,t}(x,v):=(X(s,t,x,v),V(s,t,x,v)).$$

Since $f$ is a Lagrangian solution, it is constant along the flow, i.e. for any $s,t$,
$$f(t,x,v) = f(s,X(s,t,x,v),V(s,t,x,v)).$$
Furthermore, since $(x,v)\mapsto (v,E(t,x))$ is $C^1$ and divergence-free, $\Phi_{s,t}$ is a measure-preserving $C^1$ diffeomorphism. It is straightforward to check that $\Phi_{t,s}\circ\Phi_{s,\tau}=\Phi_{t,\tau}$, which implies $\Phi_{s,t}^{-1}=\Phi_{t,s}$.\\

The next lemma is well-known (see for instance \cite{bardos_global_1985}) and essentially says that, provided $E$ is small and decays fast enough, the \emph{backwards} nonlinear flow is close to that of free transport, so one has the same lower bound on $\det\nabla_v X$ as one has for free transport, which will be used in the dispersive estimate for $\rho$.

\begin{lem}\label{geoflow}
    For all $s\leq t$ we have
    $$|\nabla_v X(s,t,x,v)-(s-t)I|\les \eps(t-s),$$
    and if $\eps$ is chosen small enough,
    $$\frac12 (t-s)^3 \leq|\det \nabla_v X(s,t,x,v)|\leq \frac32 (t-s)^3. $$
\end{lem}
\begin{proof}
    The proof follows similarly to the analysis in \cite{bardos_global_1985}. We suppress the dependence of the characteristics on $t,x,v$, so that
    \begin{equation*}
        \begin{cases}
            \nabla_v \dot{X}(s) = \nabla_v V(s), &\nabla_v X(t) = 0;\\
            \nabla_v \dot{V}(s) = \nabla_v X(s)\nabla_x E(s,X(s)),  &\nabla_v V(t) = I.
        \end{cases}
    \end{equation*}
    Then, Taylor's theorem yields 
    \begin{align*}
        \nabla_v X(s) &= \nabla_v X(t) + (s-t)\nabla_v V(t) + \int_s^t  (\tau-s)\nabla_v X(\tau)\nabla_x E(\tau,\mathcal{X}(\tau))d\tau\\
        &= (s-t)I + \int_s^t (\tau-s)\nabla_v X(\tau)\nabla_x E(\tau,X(\tau))d\tau.
    \end{align*}
    Using the bootstrap assumption for the decay of $\nabla_x E$, this implies
    \begin{align*}
        |\nabla_v X(s) - (s-t)I| &\leq \int_s^t (t-\tau)(\tau-s)|\nabla_x E(\tau,X(\tau))|d\tau + \int_s^t (\tau-s)|\nabla_v X(\tau)-(\tau-t) I||\nabla_x E(\tau,X(\tau))|d\tau\\
        &\leq \int_s^t (t-\tau)(\tau-s)\eps \jbrac{\tau}^{-3}\ln\jbrac{\tau} d\tau + \int_s^t (\tau-s)\eps \jbrac{\tau}^{-3}\ln\jbrac{\tau} |\nabla_v X(\tau)-(\tau-t) I|d\tau.
    \end{align*}
    Since $\tau \mapsto (\tau-s) \jbrac{\tau}^{-3}\ln\jbrac{\tau}$ is integrable, Gronwall yields
    \begin{align}\label{dvXsym}
        |\nabla_v X(s) - (s-t)I| &\les \eps \int_s^t (t-\tau)(\tau-s)\jbrac{\tau}^{-3}\ln\jbrac{\tau} d\tau \les \eps (t-s).
    \end{align}
    
    To obtain the lower bound on the determinant, we write 
    \begin{align*}
        |\det\nabla_v X(s)| &= \left|(t-s)^3\det\left(\big( \nabla_v X(s)(t-s)^{-1} - I\big) + I\right)\right|.
    \end{align*}
    However, $| \nabla_v X(s)(t-s)^{-1} - I| \les \eps$, so provided that $\eps$ is chosen sufficiently small, the continuity of $M\mapsto|\det M|$ gives the desired bound.
\end{proof}

Next, we will prove some estimates on derivatives of the \emph{forward} flow which will be primarily of use in constructing the asymptotic characteristics in Sections \ref{sec:asympvel} and \ref{sec:xlim}.

\begin{lem}\label{forflowdbds}
    For any $t\geq s$, there holds
    $$|\nabla_v X(t,s,x,v)|\les  (t-s), \quad |\nabla_v V(t,s,x,v)|\les 1$$
    $$|\nabla_x X(t,s,x,v)|\les 1, \quad |\nabla_x V(t,s,x,v)|\les \eps$$
\end{lem}
\begin{proof}
    The estimate on $\nabla_v X$ is essentially identical to the proof of Lemma \ref{geoflow} since the integrand in the estimate (\ref{dvXsym}) remains unchanged after switching $t$ and $s$. The estimate on $\nabla_x X$ is similar but easier. For the estimates on $\nabla_{x,v} V$, observe that from the decay estimate for $\nabla_x E$,
    \begin{align*}
        |\nabla_v V(t,s,x,v)-I|&\leq \int_s^t |\nabla_v X(\tau,s,x,v)|\|\nabla_x E(\tau)\|_{L^\infty}d\tau \\
        &\les \int_s^t \eps(\tau-s)\jbrac{\tau}^{-3} \ln\jbrac\tau \ d\tau \les \eps 
    \end{align*}
    The estimate of $\nabla_x V$ is nearly identical. 
\end{proof}

We will also need estimates on the support of the profile $g(t,x,v) = f(t,x+vt,v)$ in order to control weighted moments of $g$ and to control error terms when establishing the asymptotic behavior. To this end, denote 
\begin{align*}
    \mathcal{X}(s,t,x,v) = X(s,t,x+vt,v), \quad \mathcal{V}(s,t,x,v) = V(s,t,x+vt,v).
\end{align*}
By the method of characteristics, 
\begin{equation}\label{gcharrep}
    g(t,x,v) = f_0(\mathcal{X}(0,t,x,v),\mathcal{V}(0,t,x,v)).
\end{equation}

\begin{lem}\label{gcharbds}
    There holds 
    \begin{align*}
        \sup\{|x|:g(t,x,v)\neq 0 \text{ for some } v\} &\leq k+C\eps\ln\jbrac{t},\\
        \sup\{|v|:g(t,x,v)\neq 0 \text{ for some } x\} &\leq k+C\eps,
    \end{align*}
    where $k$ is such that $f_0$ is supported in $|x|,|v|\leq k$.
\end{lem}

\begin{proof}
    We have
    \begin{align*}
        \mathcal{X}(0,t,x,v) = x + \int_0^t (v-\mathcal{V}(s,t,x,v)) ds = x + \int_0^t \int_s^t E(\tau,\mathcal{X}(\tau,t,x,v))d\tau.
    \end{align*}
    Therefore, using the bootstrap on $E$,
    $$|\mathcal{X}(0,t,x,v)-x|\leq \int_0^t \int_s^t C\eps \jbrac{\tau}^{-2}d\tau \leq C\eps \ln\jbrac{t}.$$
    If $|x|>k+C\eps\ln\jbrac{t}$, then this implies
    $$|\mathcal{X}(0,t,x,v)|\geq |x|-|\mathcal{X}(0,t,x,v)-x|> k.$$
    Hence, if $|x|>k+C\eps\ln\jbrac{t}$, then $g(t,x,v)=0$ from the compact support of $f_0$ and the representation (\ref{gcharrep}). This gives the desired $x$ support bound. For the $v$ support bound, we estimate similarly to obtain
    $$|\mathcal{V}(0,t,x,v)-v|\leq \int_0^t C\eps\jbrac{s}^{-2}ds \leq C\eps$$
    and therefore a similar argument used to bound the $x$ support yields the result.    
\end{proof}

Finally, as a consequence of the support bounds we have the following estimate for weighted norms of $g$. These will be crucial in controlling the error terms in establishing the asymptotic behavior. 

\begin{lem}\label{vlasovweighted}
    For any $1\leq p\leq \infty$ and $0\leq q<\infty$, there holds
    $$\|\jbrac{x}^q g(t)\|_{L^p_{x,v}}\les \eps_0 \ln^{q+3/p}\jbrac{t}.$$
\end{lem}
\begin{proof}
    Using Lemma \ref{gcharbds}, $g$ is supported in $|x|\les \ln\jbrac{t}$ and $|v|\les 1$, so 
    $$\|\jbrac{x}^q g(t)\|_{L^p_{x,v}} \les \|f_0\|_{L^\infty_{x,v}}\|\ln^q\jbrac{t} \chi_{|x|\les\ln\jbrac{t}}\|_{L^p_x} \les \eps_0\ln^{q+3/p}\jbrac{t}.$$
\end{proof}

\subsection{Decay estimates for the fermions}

In this section we obtain decay estimates for $\rho$ and the potential induced by the fermions.

\begin{lem}\label{rhodecay}
    There holds
    $$\|\rho(t)\|_{L^\infty} \les \eps_0 \jbrac{t}^{-3}.$$
\end{lem}
\begin{proof}
    We make the change of variables $\mathcal{A} : v\mapsto X := X(0,t,x,v)$, which has Jacobian $|\det \nabla_v X(0,t,x,v)|\gtrsim t^3$ from Lemma \ref{geoflow}, so that
    \begin{align*}
        \rho(t,x) &= \int f_0(X(0,t,x,v),V(0,t,x,v))dv \\
        &\les t^{-3} \int  f_0(X,V(0,t,x,\mathcal{A}^{-1}(X))) dX \les \|f_0\|_{L^1_x L^\infty_v}t^{-3}.
    \end{align*}
    Of course, if $t$ is small then $\|\rho(t)\|_{L^\infty_x}\les \eps_0$ by the local well-posedness theory. 
\end{proof}

The next lemma gives decay of the potential induced by the fermions.

\begin{lem}\label{potdecayest}
    For any $r\in [0,2)$ there holds
    \begin{align*}
        \big\||\nabla|^r V*\rho\big\|_{L^\infty} \les \eps_0\jbrac{t}^{-1-r}.
    \end{align*} 
\end{lem}

\begin{proof}
    The proof is direct from the first claim of Lemma \ref{potest} with $p=1$, $q=\infty$, conservation of mass (from the volume preserving property of the characteristic flow), and Lemma \ref{rhodecay} upon noting that $||\nabla|^r V(x)|\les |x|^{-1-r}$:
    $$\big\||\nabla|^r V*\rho(t)\big\|_{L^\infty} \les \|\rho(t)\|_{L^1}^{1-(r+1)/3}\|\rho(t)\|_{L^\infty}^{(r+1)/3} \les \eps_0 \jbrac{t}^{-1-r}.$$
\end{proof}

\subsection{Asymptotic velocity characteristics}\label{sec:asympvel}

Now we construct the asymptotic velocity characteristics; the analysis here is very similar to \cite{pankavich_asymptotic_2022}. We will make use of the results in the following proposition extensively when constructing the asymptotic spatial characteristics later.

\begin{prop}\label{Vlim}
    The function $$V_\infty(s,x,v):=\lim_{t\to\infty}V(t,s,x,v) = v+\int_s^\infty E(\tau,X(\tau,s,x,v))d\tau$$ is a well-defined function in $C^{1}_{x,v}$ such that, for all $0\leq s\leq t$ and $(x,v)\in \R^6$,
    \begin{gather*}
        |V_\infty(s,x,v)-V(t,s,x,v)|\leq \eps\jbrac{t}^{-1},\\
        |\nabla_v V_\infty(s,x,v) - \nabla_v V(t,s,x,v)|\les \eps\jbrac{t}^{-1}\ln\jbrac{t},\\
        |\nabla_x V_\infty(s,x,v)-\nabla_x V(t,s,x,v)|\les \eps\jbrac{t}^{-2}\ln\jbrac{t}.
    \end{gather*}
    In particular, $V(t,s)\to V_\infty(s)$ in $C^1_{x,v}$ as $t\to\infty$. Moreover, $$|\det \nabla_v V_\infty(s,x,v)-1| \les \eps\jbrac{s}^{-1}\ln\jbrac{s}$$
    and if $\eps$ is small enough, $v\mapsto V_\infty(s,x,v)$ is bi-Lipschitz and hence injective: 
    $$|V_\infty(s,x,v_1)-V_\infty(s,x,v_2)|\geq \frac12 |v_1-v_2|.$$ 
\end{prop}
\begin{proof}
    The decay of $E$ shows that $V_\infty$ is well-defined. Moreover, the fundamental theorem of calculus yields
    \begin{align*}
        |V_\infty(s,x,v)-V(t,s,x,v)|\leq \int_t^\infty |E(\tau,X(\tau,s,x,v))|d\tau \les \eps \jbrac{t}^{-1}.
    \end{align*}
    Similarly, the decay of $\nabla_x E$ and Lemma \ref{forflowdbds} give
    \begin{align*}
        |\nabla_x V_\infty(s,x,v) - \nabla_x V(t,s,x,v)| &= \left|\int_t^\infty \nabla_x E(\tau,X(\tau))\nabla_x X(\tau) d\tau\right| \\
        &\les \int_t^\infty \eps\jbrac{\tau}^{-3}\ln\jbrac{\tau} ds \les \eps \jbrac{t}^{-2}\ln\jbrac{t}.
    \end{align*}
    and
    \begin{align*}
        |\nabla_v V_\infty(s,x,v) - \nabla_v V(t,s,x,v)| &= \left|\int_t^\infty \nabla_x E(\tau,X(\tau))\nabla_v X(\tau) d\tau\right| \\
        &\les \int_t^\infty \eps(\tau-t)\jbrac{\tau}^{-3}\ln\jbrac{\tau} ds \les \eps \jbrac{t}^{-1}\ln\jbrac{t}.
    \end{align*}
    Moreover, the fact that $V(t,s,x,v)\in C^{1}_{x,v}$ (since $E\in C^{1,\delta})$ shows that $V_\infty$ is as well. Setting $t=s$ in the above gives 
    \begin{gather*}
        |\nabla_v V_\infty(s,x,v) - I|\les \eps\jbrac{s}^{-1}\ln\jbrac{s}.
    \end{gather*}
    This implies that, provided $\eps$ is small enough, $\nabla_v V_\infty$ is an invertible matrix with $|\det \nabla_v V_\infty(s)-1|\les \eps\jbrac{s}^{-1}\ln\jbrac{s}$ since $M\mapsto\det M$ is Lipschitz. To prove injectivity, the fundamental theorem of calculus yields
    \begin{align*}
        V_\infty(s,x,v_1)-V_\infty(s,x,v_2) = \int_0^1 \nabla_v V_\infty(s,x,rv_1+(1-r)v_2)(v_1-v_2)dr.
    \end{align*}
    Therefore, 
    \begin{align*}
        |V_\infty(s,x,v_1)-V_\infty(s,x,v_2)| &\geq |v_1-v_2|-\int_0^1 |I-\nabla_v V_\infty(s,x,rv_1+(1-r)v_2)||v_1-v_2|dr\\
        &\geq |v_1-v_2|-C\eps \jbrac{s}^{-1}\ln\jbrac{s}|v_1-v_2|.
    \end{align*}
    This gives injectivity and the bi-Lipschitz property as long as $\eps$ is small enough.
\end{proof}

\begin{remark}
    The proof also shows that injectivity and the Jacobian lower bound on $V_\infty$ also hold when $s\gg 1$ rather than $\eps\ll 1$.
\end{remark}

\subsection{Asymptotic potential generated by the fermions}

In this section we complete the final ingredient needed on the Vlasov side in order to prove modified scattering for the boson subsystem: the asymptotic behavior of the potential $V*\rho$. Let us sketch the results which will be proved in this section.\\

First, using the convergence of the velocity characteristics we show that the scattering mass/spatial average $m(t,v)$ converges to a limit $m_\infty(v)$ (Proposition \ref{scatmasswk}) and as a consequence (Proposition \ref{rhoconvwk}), 
\begin{equation}\label{rhoapprox}
    \rho(t,x)\approx t^{-3} m(t,\tfrac{x}{t}) \approx t^{-3} m_\infty(\tfrac{x}{t}), \quad t\gg 1.
\end{equation}
Note that the first $\approx$ is an $=$ for free solutions by a change of variables. Due to the minimal regularity of $f_0$, these claims hold in the topology of weak convergence of measures.\\

The main result from this section is the strong $L^\infty$ convergence of $V*m(t,v)$ to $V*m_\infty(v)$ in Lemma \ref{scatmass}. Logically speaking, this is the only ingredient necessary for modified scattering of the bosons, but we include the other results because they are of independent interest and to better clarify the asymptotic behavior. As a consequence, we also show that 

$$(V*\rho)(t,x) \approx t^{-1}(V*m_\infty)(\tfrac{x}{t}), \quad t\gg 1,$$

as suggested by (\ref{rhoapprox}). This serves as the key motivation behind why $\ln(t)(V*m_\infty)(\tfrac{x}{t})$ appears in the long-range phase correction.\\

First we show the convergence towards $m_\infty$, proving (\ref{masymp}).

\begin{prop}\label{scatmasswk}
    There exists $m_\infty \in L^\infty$ with compact support such that
    $$m(t,v)dv\wkly m_\infty(v)dv \quad \text{ as }t\to\infty,$$
    weakly in the sense of measures. 
\end{prop}
\begin{proof}
Let $\varphi\in C_c(\R^3)$. Using the measure preserving property of the flow,
    \begin{align*}
        \int \varphi(v)m(t,v)dv &= \iint \varphi(v)f(t,x,v)dxdv\\
        &= \iint \varphi(v)f_0(X(0,t,x,v),V(0,t,x,v))dxdv\\
        &= \iint \varphi(V(t,0,x,v))f_0(x,v)dxdv.
    \end{align*}
    Let $\mathbb{V}_x$ be the map $v\mapsto V_\infty(0,x,v)$. Then, using Proposition \ref{Vlim}, namely using the convergence of $V(t)$ to $V_\infty$
    and the invertibility of $\mathbb{V}_x$, we get
    \begin{align*}
        \lim_{t\to\infty}\int \varphi(v)m(t,v)dv &= \iint \varphi(V_\infty(0,x,v)) f_0(x,v)dxdv\\
        &= \iint \varphi(v)f_0(x,\mathbb{V}_x^{-1}(v))\big|\det (\nabla_v\mathbb{V}_x)(\mathbb{V}_x^{-1}(v))|^{-1}dxdv.
    \end{align*}
    Therefore, we obtain $m(t,v)dv\wkly m_\infty(v)dv$, with 
    \begin{equation}\label{minfdef}
        m_\infty(v)=\int f_0(x,\mathbb{V}_x^{-1}(v))\big|\det (\nabla_v\mathbb{V}_x)(\mathbb{V}_x^{-1}(v))|^{-1}dx.
    \end{equation}
    By the compact $x$ support of $f_0$, the fact $f_0\in L^\infty$, and the Jacobian lower bound in Proposition \ref{Vlim}, we have $m_\infty\in L^\infty$. Moreover, $m_\infty$ is compactly supported as $\mathbb{V}_x$ is bi-Lipschitz so its inverse maps bounded sets to bounded sets. 
\end{proof}

 The next result essentially says that 
$$\rho(t,x)\approx t^{-3}m_\infty\big(\tfrac{x}{t}\big), \quad t\gg 1,$$
and proves (\ref{rhoasymp}).

\begin{prop}\label{rhoconvwk}
    There holds 
    $$t^3\rho(t,tv)dv\wkly m_\infty(v)dv \quad \text{ as } t\to\infty,$$
    weakly in the sense of measures.
\end{prop}
\begin{proof}
Let $\varphi\in C_c(\R^3)$. A change of variables $(u,v)\mapsto(x,v):=(t(v-u),v)$ yields
\begin{align*}
    \int \varphi(v)t^3\rho(t,tv)dv &= \iint \varphi(v)t^3 g(t,t(v-u),u)dudv\\
    &= \iint \varphi(v)g(t,x,v-\tfrac{x}{t})dxdv\\
    &= \iint \varphi(v+\tfrac{x}{t})g(t,x,v)dxdv.
\end{align*}
Therefore, using the support bounds from Lemma \ref{gcharbds} we get
\begin{align*}
    \left|\int \varphi(v)\big( t^3\rho(t,tv)-m(t,v)\big) \right| &\leq \iint \big|\varphi(v+\tfrac{x}{t})-\varphi(v) \big|g(t,x,v)dxdv\\
    &\leq \|f_0\|_{L^1_{x,v}}\sup_{\substack{|x|\leq C\ln\jbrac{t}\\ |v|\leq C}}\big|\varphi(v+\tfrac{x}{t})-\varphi(v) \big| \to 0 \quad \text{ as } t\to\infty,
\end{align*}
in view of the uniform continuity of $\varphi$. Finally, the weak convergence of $m(t,v)$ to $m_\infty(v)$ proved in Proposition \ref{scatmasswk} yields the result.
\end{proof}

Due to the smoothing effect of the Coulomb potential, we can prove the convergence of $V*m(t,v)$ to $V*m_\infty$ in $L^\infty_v$ rather than in a weak topology. This will be used to construct the phase correction in the proof of Proposition \ref{profilebd} later on.
\begin{lem}\label{scatmass}
    The function $V*m_\infty\in L^\infty$ satisfies
    $$\|V*m(t)-V*m_\infty\|_{L^\infty_v} \les  \eps\eps_0 \jbrac{t}^{-1}.$$
\end{lem}

\begin{proof}
    We have
    \begin{gather*}
        V*m(t,v) = \iint V(v-u)f(t,x,u)dxdu = \iint V(v-V(t,0,x,u))f_0(x,u)dxdu\\
        V*m_{\infty}(v) = \iint V(v-V_\infty(0,x,u)) f_0(x,u)dxdu.
    \end{gather*}
    The second identity follows from changing variables in (\ref{minfdef}). Define
    $$I=\iint \big|V(v-V_\infty(0,x,u))-V(v-V(t,0,x,u))\big|f_0(x,u)dxdu.$$
    We will split into regions near and far from the singularities. Let $R=2\|V_\infty-V(t)\|_{L^\infty_{x,v}}$ and let $S$ be the near region:
    $$S:=\{(x,u): |v-V(t,0,x,u)|,|v-V_\infty(0,x,u)|\leq R\}.$$
    Also, let $I_1$ be the integral over $S$ and $I_2$ the integral over $S^c$. For $I_1$ we estimate directly, using the compact support of $f_0$ and making the change of variables $w=v-V_\infty(0,x,u)$ in the first term and $w=v-V(t,0,x,u)$ in the second (recalling that both transformations are uniformly bi-Lipschitz by Proposition \ref{Vlim}) to obtain
    \begin{align*}
        I_1&\les \|f_0\|_{L^\infty}\sup_x \Big(\int_{\{u:|v-V_\infty(0,x,u)|\leq R\}}\frac{du}{|v-V_\infty(0,x,u)|}\\
        &\qquad\qquad\qquad\quad+\int_{\{u:|v-V(t,0,x,u)|\leq R\}}\frac{du}{|v-V(t,0,x,u)|}\Big)\\
        &\les \eps_0 \int_{|w|\leq R} \frac{dw}{|w|} \les \eps_0 R^2.
    \end{align*}
    For $I_2$, note that on $S^c$ we either have
    $$|v-V(t,0,x,u)|\geq |v-V_{\infty}(0,x,u)|-\|V_\infty-V(t)\|_{L^\infty_{x,v}} \geq \tfrac12 |v-V_{\infty}(0,x,u)|$$
    or
    $$|v-V_\infty(0,x,u)|\geq |v-V(t,0,x,u)|-\|V_\infty-V(t)\|_{L^\infty_{x,v}} \geq \tfrac12 |v-V(t,0,x,u)|.$$
    Without loss of generality we can focus on the first case $|v-V_\infty(0,x,u)|>R$ as the proof for the other case will be the same. Therefore,
    \begin{align*}
        I_2&\leq \iint_{\{(x,u):|v-V_\infty(0,x,u)|>R\}} \frac{\big||v-V(t,0,x,u)|-|v-V_{\infty}(0,x,u)|\big|}{|v-V_{\infty}(0,x,u)||v-V(t,0,x,u)|}f_0(x,u)dxdu\\
        &\les \eps_0 R \sup_x \int_{\{u:f_0(x,u)\neq 0\}}\frac{du}{|v-V_\infty(0,x,u)|^2}.
    \end{align*}
    If we change variables $w=v-V_\infty(0,x,u)$, then since $v\mapsto V_\infty$ is bi-Lipschitz and $f_0$ has compact support, the integral above is over a compact set. But $|w|^{-2}$ is locally integrable and so we obtain $I_2\les \eps_0 R$. Finally, recalling Proposition \ref{Vlim}, we arrive at
    $$|V*m(t,v)-V*m_\infty(v)|\leq I \les \eps_0 R\les \eps\eps_0 \jbrac{t}^{-1}.$$
\end{proof}

Finally, we prove (\ref{VrhoSSasymp}).

\begin{prop}\label{asympfermpot}
    There holds 
    $$\|V*\rho(t,x)-t^{-1}(V*m_\infty)(\tfrac{x}{t})\|_{L^\infty_x}\les \eps\eps_0 t^{-2} \ln^5\jbrac{t}$$
\end{prop}
\begin{proof}
    In view of Lemma \ref{scatmass}, it suffices to estimate $\|V*\rho(t,x) - t^{-1}(V*m)(t,\tfrac{x}{t})\|_{L^\infty_x}$. By a change of variables,
$$\rho(t,x) = \int g(t,x-vt,v)dv = t^{-3}\int g(t,y,\tfrac{x-y}{t})dy.$$
Hence,
$$(V*\rho)(t,x) = t^{-3}\iint V(x-z)g(t,y,\tfrac{z-y}{t})dydz = t^{-3}\iint V(x-z-y)g(t,y,\tfrac{z}{t})dydz,$$
and using the $-1$ homogeneity of $V$,
\begin{align*}
    t^{-1}(V*m)(t,\tfrac{x}{t}) &= t^{-1}\iint V(\tfrac{x}{t}-v)g(t,y,v)dydv = t^{-3}\iint V(x-z)g(t,y,\tfrac{z}{t})dydz.
\end{align*}

Therefore, 
    \begin{align*}
        |(V*\rho)(t,x)-t^{-1}(V*m)(t,\tfrac{x}{t})|\leq t^{-3}\iint |V(x-z-y)-V(x-z)|g(t,y,\tfrac{z}{t}) dydz .
    \end{align*}
 When $y,z$ is such that $|x-z|>2|y|$, we have $|x-z-y|\geq |x-z|-|y| > \frac12|x-z|$, so using Lemmas \ref{potest} and \ref{gcharbds} gives, with $m_1(t,v)=\int |y|g(t,y,v)dy$, 
 \begin{align*}
     \iint_{|x-z|>2|y|} &|V(x-z-y)-V(x-z)|g(t,y,\tfrac{z}{t}) dydz \\
     &\les \iint_{|x-z|>2|y|} \frac{|y|}{|x-z||x-z-y|}g(t,y,\tfrac{z}{t}) dydz \\
     &\les \big\||\cdot|^{-2}*m_1(t,\tfrac{\cdot}{t})\big\|_{L^\infty_x}\\
     &\les \big(t^3\|m_1(t)\|_{L^1}\big)^{1/3} \|m_1(t)\|_{L^\infty}^{2/3}\\
     &\les (t^3 \eps_0\ln^{4}\jbrac{t})^{1/3}(\eps_0\ln^4\jbrac{t})^{2/3} = \eps_0 t\ln^{4}\jbrac{t}.
 \end{align*}
 This gives the desired bound when $|x-z|>2|y|$. On the other hand, when $|x-z|\leq 2|y|$, we just bound directly 
 \begin{align*}
     \iint_{|x-z|\leq 2|y|} &|V(x-z-y)-V(x-z)|g(t,y,\tfrac{z}{t}) dydz \\
     &\leq  \iint_{|x-z|\leq 2|y|}\left( \frac{1}{|x-z-y|}+ \frac{1}{|x-z|}\right) g(t,y,\tfrac{z}{t}) dydz\\
     &\leq \eps_0\iint_{\{|x-z|\leq 2|y|\}\cap \{y:g(t,y,\frac{z}{t})\neq 0\}} \frac{1}{|x-z-y|}+ \frac{1}{|x-z|} dzdy.
 \end{align*}
 For the first term inside the integral, we use Lemma \ref{gcharbds} and the fact that $|x-y-z|\leq 3|y|$ to obtain
 \begin{align*}
     \iint_{\{|x-z|\leq 2|y|\}\cap\{y:g(t,y,\frac{z}{t})\neq 0\}} \frac{1}{|x-z-y|} dzdy &\leq \int_{|y|\les\ln\jbrac{t}}\int_{|x-z-y|\les \ln\jbrac{t}} \frac{1}{|x-z-y|} dzdy \\
     &\les \ln^3\jbrac{t} \int_0^{C\ln\jbrac{t}}rdr \les \ln^5\jbrac{t}.
 \end{align*}
 The estimate of the second term inside the integral is nearly identical, but easier. This gives the desired bounds.
\end{proof}

\section{Analysis of the bosonic part}\label{boson}

\subsection{Hartree energy estimates}

Recall that the vector field $J := x+it\nabla_x$ commutes with $i\p_t + \frac12 \Delta_x$, and moreover satisfies, with $M(t)$ the multiplication operator by $e^{it|x|^2/2t}$, 
$$J = M(t) (it\nabla) M(-t), \quad J = e^{it\Delta/2}xe^{-it\Delta/2}.$$
The first identity leads to the definition
$$|J|^s = t^s M(t)|\nabla|^s M(-t).$$
with $|\nabla|^s$ the Fourier multiplier operator with symbol $|\xi|^s$.

\begin{prop}\label{EEhartree}
There holds
    \begin{align*}
        \|\phi(t)\|_{H^s} &\les \eps_0 (1+ \eps), \\
        \big\||J|^{s}\phi(t)\big\|_{L^2} &\les \eps_0 (1+ \eps)\ln^2\jbrac{t}.
    \end{align*}
\end{prop}

\begin{proof}

Let $\Lambda\in \{\text{Id},J,|\nabla|^s,|J|^s\}$. Since these all commute with the linear Schr\"odinger operator, we have
\begin{align}\label{schenergyid}
    \|\Lambda\phi(t)\|_{L^2}^2 = \|\Lambda\phi(0)\|_{L^2}^2 + 2\int_0^t \Imag \jbrac{ \Lambda [(V*\rho)\phi](\tau), \Lambda \phi(\tau) }d\tau.
\end{align}

When $\Lambda = \text{Id}$, we get conservation of mass:
$$\|\phi(t)\|_{L^2}^2 = \|\phi_0\|_{L^2}^2 + 2\Imag \int_0^t (V*\rho(\tau))|\phi(\tau)|^2 d\tau = \|\phi_0\|_{L^2}^2 \leq \eps_0^2.$$
When $\Lambda = \nabla$, we have
\begin{align*}
    \|\nabla\phi(t)\|_{L^2}^2 &= \|\nabla\phi_0\|_{L^2}^2 + 2\int_0^t \Imag \jbrac{ (\nabla V*\rho)\phi + (V*\rho)\nabla\phi, \nabla \phi }d\tau\\
    &\leq \|\nabla\phi_0\|_{L^2}^2 + C\int_0^t\|\nabla V*\rho(\tau)\|_{L^\infty}\|\phi(\tau)\|_{L^2}\|\nabla\phi(\tau)\|_{L^2}d\tau.
\end{align*}
Therefore, Gronwall, conservation of mass, and Lemma \ref{potdecayest} yield
$$\|\nabla\phi(t)\|_{L^2}\leq \|\nabla\phi_0\|_{L^2} + C\int_0^t \eps\eps_0\jbrac{\tau}^{-2}d\tau \les \eps_0 + \eps\eps_0.$$

Similarly, when $\Lambda = J$, 
\begin{align*}
    \|J\phi(t)\|_{L^2}^2 &= \|x\phi_0\|_{L^2}^2 + 2\int_0^t \Imag \jbrac{ (V*\rho)J\phi + i\tau(\nabla V*\rho)\phi, J \phi }d\tau\\
    &\leq \|x\phi_0\|_{L^2}^2 + C\int_0^t \tau\|\nabla V*\rho(\tau)\|_{L^\infty}\|\phi(\tau)\|_{L^2}\|J\phi(\tau)\|_{L^2}d\tau.
\end{align*}
So Gronwall and Lemma \ref{potdecayest} yields
$$\|J\phi(t)\|_{L^2}\leq \|x\phi_0\|_{L^2} + C\int_0^t \eps\eps_0\jbrac{\tau}^{-1}d\tau \les \eps_0 + \eps\eps_0\ln\jbrac{t}.$$

When $\Lambda = |\nabla|^s, |J|^s$, we need to use commutator estimates. Let $s=1+\sigma$ where $\sigma\in (1/2,1)$, and write $|\nabla|^s = R\cdot|\nabla|^\sigma\nabla$, where $R$ is the vector-valued Riesz transform. Then
\begin{align}\label{commid}
    |\nabla|^s (uv) &= R\cdot |\nabla|^\sigma(u\nabla v + v\nabla u)\notag \\
    &= R\cdot \big( u|\nabla|^\sigma\nabla v + v|\nabla|^\sigma \nabla u + C^\sigma(u,\nabla v) + C^\sigma(v,\nabla u)\big),
\end{align}
where
$$C^\sigma(w_1,w_2) = |\nabla|^\sigma (w_1w_2) - w_1|\nabla|^\sigma w_2.$$

Recall the following commutator estimate from Theorem 5.1 of \cite{li_kato-ponce_2019}:
 \begin{equation}\label{commest}
     \big\|C^\sigma(w_1,w_2)\big\|_{L^p} \les \big\||\nabla|^{\sigma}w_1\big\|_{L^q}\|w_2\|_{L^r},
 \end{equation}
 for all $1<p<\infty, 1<q,r\leq \infty$ and $\sigma>0$ satisfying $1/p = 1/q + 1/r$.\\

We use (\ref{commid}) with $u=V*\rho$ and $v=\phi$. When all derivatives fall on $\phi$, the fact that $R$ is $L^2$-skew adjoint with $R\cdot R = -\text{Id}$ and the identity $|\nabla|^s = R\cdot|\nabla|^\sigma\nabla$ gives
\begin{align*}
    \Imag \jbrac{ R\cdot((V*\rho)|\nabla|^\sigma\nabla\phi), |\nabla|^s\phi } = \Imag\jbrac{ (V*\rho)|\nabla|^\sigma\nabla\phi, |\nabla|^\sigma\nabla\phi } = 0.
\end{align*}
Therefore, the top order term in the energy estimate vanishes, which is responsible for logarithmic rather than slow algebraic growth of the norms. Using (\ref{commest}), Lemma \ref{potdecayest}, the bootstrap assumptions, and interpolation, 
\begin{align*}
    \|\phi(t)\|_{\dot{H}^s}^2 &\les \eps_0^2 + \int_0^t \Big(\big\||\nabla|^\sigma \nabla V*\rho\big\|_{L^\infty} \|\phi\|_{L^2} + \big\||\nabla|^\sigma V*\rho\big\|_{L^\infty} \|\nabla \phi\|_{L^2} +  \\
    & \qquad\qquad\qquad\quad +\|\nabla V*\rho\|_{L^\infty} \big\||\nabla|^\sigma\phi\|_{L^2}\Big) \|\phi\|_{\dot{H}^s} d\tau\\
    &\les \eps_0^2 + \int_0^t \eps\eps_0 \big( \jbrac{\tau}^{-1-s} + \jbrac{\tau}^{-1-\sigma} + \jbrac{\tau}^{-2} \big)\|\phi\|_{\dot{H}^s}d\tau.
\end{align*}
Therefore, Gronwall gives 
$$\|\phi(t)\|_{\dot{H}^s} \les \eps_0 + \eps\eps_0.$$

Estimating $|J|^s$ is very similar. Using (\ref{schenergyid}), we have, with $w = M(-t)\phi$,
\begin{align*}
    \big\||J|^s\phi(t)\big\|_{L^2}^2 = \big\||x|^s\phi_0\big\|_{L^2}^2+2\int_0^t t^{2s}\Imag\jbrac{ |\nabla|^s[(V*\rho)w](\tau),|\nabla|^s w(\tau)} d\tau.
\end{align*}
Proceeding in the same manner as the estimate of $\phi$ in $H^s$, namely, using (\ref{commid}), (\ref{commest}), and Lemma \ref{potdecayest}, we get 
\begin{align*}
    \big\||J|^s\phi(t)\big\|_{L^2}^2 &\les \eps_0^2 + \int_0^t \tau^{2s}\Big(\big\||\nabla|^\sigma \nabla V*\rho\big\|_{L^\infty} \|w\|_{L^2} + \big\||\nabla|^\sigma V*\rho\big\|_{L^\infty} \|\nabla w\|_{L^2} +  \\
    & \qquad\qquad\qquad\quad +\|\nabla V*\rho\|_{L^\infty} \big\||\nabla|^\sigma w\|_{L^2}\Big) \big\||\nabla|^s w\big\|_{L^2} d\tau\\
    &= \eps_0^2 + \int_0^t \tau^{s}\Big(\big\||\nabla|^\sigma \nabla V*\rho\big\|_{L^\infty} \|\phi\|_{L^2} + \big\||\nabla|^\sigma V*\rho\big\|_{L^\infty} \tau^{-1}\|J\phi\|_{L^2} +  \\
    & \qquad\qquad\qquad\quad +\|\nabla V*\rho\|_{L^\infty} \tau^{-\sigma}\big\||J|^\sigma \phi\|_{L^2}\Big) \big\||J|^s \phi\big\|_{L^2} d\tau\\
    &\les \eps_0^2 + \int_0^t \eps\eps_0\tau^{s}\big( \jbrac{\tau}^{-1-s} + \jbrac{\tau}^{-2-\sigma}
    \ln\jbrac{\tau} + \jbrac{\tau}^{-2-\sigma}\ln\jbrac{\tau}\big)\big\||J|^s\phi\big\|_{L^2}d\tau\\
    &\les \eps_0^2 + \int_0^t \eps\eps_0 \jbrac{\tau}^{-1}\ln\jbrac{\tau}\big\||J|^s\phi\|_{L^2}d\tau.
\end{align*}
In the third line we used interpolation to obtain $\||J|^\sigma \phi\|_{L^2}\les \|\phi\|_{L^2}+\|J\phi\|_{L^2} \les \eps\ln\jbrac{t}$. As a result, Gronwall yields

$$\big\||J|^s\phi(t)\big\|_{L^2} \les \eps_0+\eps\eps_0\ln^{2}\jbrac{t},$$
as desired. 
\end{proof}

\subsection{Decay estimates for the bosons}\label{hartreedecay}

First, recall the Fraunhofer formula, which gives a refined dispersive estimate for the Schr\"odinger semigroup and will enable us to prove the asymptotic expansions for the fields.

\begin{lem}\label{fraunhofer}
    There holds
    $$(e^{it\Delta/2}h)(t,x) = \frac{e^{i|x|^2/2t}}{(it)^{3/2}}  \widehat{h}(t,x/t) + R(t,x),$$
    where the remainder satisfies
    \begin{align*}
        \|R(t)\|_{L^\infty_x} &\les t^{-3/2-\delta}\|\jbrac{x}^s h(t,x)\|_{L^2_x},\\
        \|R(t)\|_{L^2_x} &\les t^{-s/2}\|\jbrac{x}^s h(t,x)\|_{L^2_x},
    \end{align*}
    
    for any $\delta<\frac12 s_{gap}$. As a consequence, taking $h$ to be the profile $\psi=e^{-it\Delta/2}\phi$,
    
    $$\|\phi(t)\|_{L^\infty} \les t^{-3/2}\|\hat{\psi}(t,\xi)\|_{L^\infty_\xi} + t^{-3/2 - \delta}\|\jbrac{x}^s\psi(t,x)\|_{L^2_x}.$$
\end{lem}
\begin{proof}
    The proof is well-known, but the techniques serve as an important warm-up to the ideas that follow, so we include it. We expand the phase and extract the leading order term by replacing the slowly oscillating factor $e^{i|y|^2/2t}$ with $1$:
    \begin{align*}
        e^{it\Delta/2}h(t,x) &= (it)^{-3/2}\int e^{i|x-y|^2/2t}h(t,y)dy\\
        &= (it)^{-3/2}e^{i|x|^2/2t}\left(\int  e^{-iy\cdot x/t} h(t,y)dy + \int e^{-iy\cdot x/t}\big(e^{i|y|^2/2t}-1\big) h(t,y)dy \right)\\
        &= (it)^{-3/2}e^{i|x|^2/2t} \widehat{h}(t,x/t) + R(t,x).
    \end{align*}
    We observe that
    $$R(t,x) = (it)^{-3/2}e^{i|x|^2/2t}\mathcal{F}_y\big[(e^{i|y|^2/2t}-1)h(t,y) \big](x/t).$$
    Therefore, using Hausdorff-Young and the inequality $|e^{ix}-1| \leq 2|x|^{\delta}$,
    \begin{align*}
        \|R(t)\|_{L^\infty_x} \leq t^{-3/2}\|(e^{i|y|^2/2t}-1)h(t,y)\|_{L^1_y} \les t^{-3/2-\delta} \||y|^{2\delta} h(t,y)\|_{L^1_y}.
    \end{align*}
    Finally, by using Cauchy-Schwartz, we obtain the desired bound since $s-2\delta>3/2$, which is exactly the condition $\delta<\frac12\Big(s-\frac32\Big)=\frac12 s_{gap}$:
    \begin{align*}
        \||y|^{2\delta} h(t,y)\|_{L^1_y} \les \|\jbrac{y}^{-s+2\delta}\|_{L^2_y}\|\jbrac{y}^s h(t,y)\|_{L^2} \les \|\jbrac{y}^s h(t,y)\|_{L^2}.
    \end{align*}
    The $L^2$ estimate follows similarly by using Plancherel: 
    \begin{align*}
        \|R(t)\|_{L^2_x} = t^{-3/2}\|(e^{i|x/t|^2/2t}-1)h(t,x/t)\|_{L^2_x} = \|(e^{i|y|^2/2t}-1)h(t,y)\|_{L^2_y} \les t^{-s/2}\|\jbrac{y}^s h(t,y)\|_{L^2}.
    \end{align*}
\end{proof}

Motivated by the previous lemma, in the next section we will prove the following proposition, which will complete the bootstrap argument.

\begin{prop}\label{profilebd}
    We have 
    $$\|\hat\psi(t)\|_{L^\infty_\xi}\les \eps_0(1 + \eps).$$
\end{prop}

Granting this proposition for now, let us complete the proof of Proposition \ref{bootstrap} and thereby close the bootstrap argument for the decay estimates. First, (\ref{phidecaybs}) holds:

\begin{lem}\label{schrdecay}
    There holds 
    $$\|\phi(t)\|_{L^\infty}\les \eps_0(1 + \eps)\jbrac{t}^{-3/2}.$$
\end{lem}
\begin{proof}
In view of Proposition \ref{profilebd}, it suffices to bound the remainder $R(t,x)$ in Lemma \ref{fraunhofer}. Recall that $|J|^s = e^{it\Delta/2}|x|^se^{-it\Delta/2}$, so by unitarity of $e^{it\Delta/2}$ and Proposition \ref{EEhartree}, 
$$\|\jbrac{x}^s\psi\|_{L^2} \les \|\psi\|_{L^2}+\||x|^s \psi\|_{L^2}= \|\phi\|_{L^2}+\||J|^s \phi\|_{L^2}\les \eps_0 + \eps\eps_0\ln^2\jbrac{t}.$$
This gives the desired decay for $R$. Finally, the bound near $t=0$ follows from Sobolev embedding $H^s\hookrightarrow L^\infty$ and Proposition \ref{EEhartree}.
\end{proof}

Lemma allows us to close the bootstrap for the decay of $E$, yielding (\ref{improvedbs}) and thereby completing the proof of Proposition \ref{bootstrap}:
\begin{lem}\label{Edecay} 
There holds
    $$\|E\|_{L^\infty} \les \eps_0^2(1+\eps)^{4/3} \jbrac{t}^{-2}, \quad \|\nabla_x E\|_{L^\infty} \les \eps_0^2(1+\eps)^2\jbrac{t}^{-3}\ln\jbrac{t}.$$
\end{lem}
\begin{proof}
    The decay estimate for $E$ follows directly from the first claim of Lemma \ref{potest} with $q=1, p=\infty$, noting that $|\nabla V|\les |x|^{-2}$:
    $$\|E\|_{L^\infty} \les \||\phi|^2\|_{L^1}^{1/3}\||\phi|^2\|_{L^\infty}^{2/3}\les \eps_0^{2/3}(\eps_0(1+\eps)\jbrac{t}^{-3/2})^{4/3} = \eps_0^2(1+\eps)^{4/3}\jbrac{t}^{-2}.$$
    For $\|\nabla_x E\|_{L^\infty}$, we use the second claim of Lemma \ref{potest}, taking $r=\jbrac{t}^{-3/2s_{gap}}$, $R=\jbrac{t}$, and $p=\frac{6}{5-2s}>3$, so using the bootstrap assumptions, Sobolev embedding $\dot{H}^{s-1}\hookrightarrow L^{\frac{6}{3-2(s-1)}}=L^{\frac{6}{5-2s}}$ and $H^s\hookrightarrow L^\infty$, and Lemma \ref{schrdecay} yields
    \begin{align*}
        \|\nabla_x E\|_{L^\infty} &\les \jbrac{t}^{-\frac{3}{2s_{gap}}(1-\frac{3(5-2s)}{6})}\|\nabla_x|\phi|^2\|_{L^{\frac{6}{5-2s}}} + (1+\ln\jbrac{t}) \||\phi|^2\|_{L^\infty} + \jbrac{t}^{-3}\||\phi|^2\|_{L^1}\\
        &\les \jbrac{t}^{-3/2}\|\phi\|_{L^\infty}\|\nabla_x\phi\|_{L^{\frac{6}{5-2s}}} + \ln\jbrac{t}\|\phi\|_{L^\infty}^2 + \jbrac{t}^{-3}\|\phi\|_{L^2}^2\\
        &\les \eps_0(1+\eps)\jbrac{t}^{-3}\|\phi\|_{H^s} + \eps_0^2(1+\eps)^2 \jbrac{t}^{-3}\ln\jbrac{t}\\
        &\les \eps_0^2(1+\eps)^2\jbrac{t}^{-3}\ln\jbrac{t}.
    \end{align*}
\end{proof}

\subsection{Control of $\|\hat\psi\|_{L^\infty_\xi}$}

Now we prove Proposition \ref{profilebd}, which requires us to identify the leading order dynamics and extract a phase correction. This will form the basis for establishing modified scattering in the next section. In essence, the proof is a more sophisticated version of the proof of Lemma \ref{fraunhofer} and is inspired by Kato and Pusateri's proof of modified scattering for Hartree in \cite{kato_new_2011}.

\begin{proof}[Proof of Proposition \ref{profilebd}]
 
Since $\hat{f}(s,\eta,v) = e^{-ivs\cdot\eta}\hat{g}(s,\eta,v)$ and $\hat{\phi}(s,\xi-\eta) = e^{-it|\xi-\eta|^2/2}\hat\psi(s,\xi-\eta)$, the Duhamel formula, the fact that $\hat{V}(\xi)=|\xi|^{-2}$, the change of variables $v\mapsto \xi-v$, and Plancherel yield
\begin{align*}
    \hat{\psi}(t,\xi) &= \hat{\phi}_0(\xi) - i\int_0^t \int e^{is|\xi|^2/2}|\eta|^{-2}\hat{\rho}(s,\eta)\hat{\phi}(s,\xi-\eta) d\eta ds\\
    &=\hat{\phi}_0(\xi) - i\int_0^t \iint e^{is[(\xi-v)\cdot\eta - |\eta|^2/2]} |\eta|^{-2}\hat{g}(s,\eta,v)\hat{\psi}(s,\xi-\eta) dv d\eta ds\\
    &= \hat{\phi}_0(\xi) - i\int_0^t \iint e^{is(v\cdot\eta - |\eta|^2/2)} |\eta|^{-2}\hat{g}(s,\eta,\xi-v)\hat{\psi}(s,\xi-\eta) dv d\eta ds\\
    &= \hat{\phi}_0(\xi) - i \int_0^t \iint \mathcal{F}_{v,\eta}[e^{is(v\cdot\eta - |\eta|^2/2)}|\eta|^{-2}]\mathcal{F}^{-1}_{v,\eta}[\hat{g}(s,\eta,\xi-v)\hat{\psi}(s,\xi-\eta)] dv d\eta ds.
\end{align*}

Using the $-3$ homogeneity of the Dirac mass, we compute and extract the factor $s^{-1}$, which makes the phase slowly varying as $s\to\infty$:
\begin{align*}
    \mathcal{F}_{v,\eta}[e^{is(v\cdot\eta-|\eta|^2/2)} |\eta|^{-2}] &= \iint e^{-iv\cdot v' - i\eta\cdot \eta'} e^{is(v'\cdot\eta'-|\eta'|^2/2)}|\eta'|^{-2}  dv' d\eta'\\
    &= (2\pi)^3 \int \delta(v-s\eta') e^{-i\eta\cdot\eta'}e^{-is|\eta'|^2/2}|\eta'|^{-2}d\eta'\\
    &=  (2\pi)^3 s^{-3}\int \delta(v/s-\eta') e^{-i\eta\cdot\eta'}e^{-is|\eta'|^2/2}|\eta'|^{-2}d\eta'\\
    &= (2\pi)^3 s^{-1} e^{-i\eta\cdot v/s} e^{-i|v|^2/2s}|v|^{-2}.
\end{align*}
Therefore,
\begin{align*}
    \hat{\psi}(t,\xi) &= \hat{\phi}_0(\xi) -i(2\pi)^{3} \int_0^t \iint  s^{-1} e^{-i(|v|^2/2+\eta\cdot v)/s}|v|^{-2} \mathcal{F}^{-1}_{v,\eta}[\hat{g}(s,\eta,\xi-v)\hat{\psi}(s,\xi-\eta)] dv d\eta ds.
\end{align*}

Since the phase is small when $s$ is large, to extract the leading term $\mathcal{L}$ we replace the phase factor with $1$ and introduce a remainder $\mathcal{R}$, similar to  \cite{kato_new_2011} and the proof of Lemma \ref{fraunhofer}:
\begin{gather*}
    \hat{\psi}(t,\xi) = \hat{\phi}_0(\xi) -i \int_0^t \mathcal{L}(s,\xi) + \mathcal{R}(s,\xi) ds,\\
    \mathcal{L}(s,\xi) = (2\pi)^3 s^{-1}\iint |v|^{-2}\mathcal{F}^{-1}_{v,\eta}[\hat{g}(s,\eta,\xi-v)\hat{\psi}(s,\xi-\eta)] dv d\eta,\\
    \mathcal{R}(s,\xi) = (2\pi)^3 s^{-1}\iint |v|^{-2}[e^{-i(|v|^2/2+v\cdot\eta)/s} -1]\mathcal{F}^{-1}_{v,\eta}[\hat{g}(s,\eta,\xi-v)\hat\psi(s,\xi-\eta)]dv d\eta.
\end{gather*}

We compute the leading term using Plancherel:
\begin{align*}
    \mathcal{L}(s,\xi) &= (2\pi)^{3} s^{-1}\iint |v|^{-2}\mathcal{F}^{-1}_{v,\eta}[\hat{g}(s,\eta,\xi-v)\hat{\psi}(s,\xi-\eta)] dv d\eta \\
    &= (2\pi)^3 s^{-1} \iint \mathcal{F}_v^{-1}[|v|^{-2}]\mathcal{F}_\eta^{-1}[\hat{g}(s,\eta,\xi-v)\hat{\psi}(s,\xi-\eta)]dv d\eta\\
    &= (2\pi)^{-3} s^{-1}\int (4\pi|v|)^{-1} \iint e^{i\eta\cdot\eta'}\hat{g}(s,\eta',\xi-v)\hat{\psi}(s,\xi-\eta')d\eta' d\eta dv\\
    &= s^{-1}\int (4\pi|v|)^{-1}\int \hat{g}(s,\eta',\xi-v)\hat{\psi}(s,\xi-\eta')\delta(\eta')d\eta' dv\\
    &=  s^{-1}\hat\psi(s,\xi)\int (4\pi|v|)^{-1} \hat{g}(s,0,\xi-v)dv\\
    &=  s^{-1}\hat{\psi}(s,\xi)(V*_\xi \hat{g})(s,0,\xi)\\
    &= s^{-1}\hat\psi(s,\xi) \int (V*_\xi g)(s,x,\xi)dx\\
    &= s^{-1}\hat\psi(s,\xi) (V* m)(s,\xi).
\end{align*}

This suggests introducing an integrating factor in order to remove the leading term, leaving only the remainder, which, in analogy to Lemma \ref{fraunhofer}, we expect to decay faster. With this in mind, we define

$$P(t,\xi) := \int_1^t s^{-1}(V* m)(s,\xi) ds, \quad B(t,\xi) := e^{iP(t,\xi)}.$$

Since
$$\p_t\hat\psi(t,\xi) = -it^{-1}\hat\psi(t,\xi) (V* m)(t,\xi) - i\mathcal{R}(t,\xi),$$
the equation for $\hat{w}(t,\xi) := B(t,\xi)\hat\psi(t,\xi)$ satisfies
$$\p_t\hat w(t,\xi) = - iB(t,
\xi)\mathcal{R}(t,\xi).$$
Crucially, $|B(t,\xi)|=1$ since the phase of $B$ is purely imaginary, so bounding $\|\hat{\psi}(t)\|_{L^\infty}$ is equivalent to bounding $\|\hat{w}(t)\|_{L^\infty}$. Therefore, we can reduce the proof of the present proposition to the following lemma:
\begin{lem}\label{remest}
    There holds
    \begin{align*}
        \|\mathcal{R}(t)\|_{L^2_\xi \cap L^\infty_\xi}&\les \eps_0^2(1+\eps) t^{-1-\delta}, \quad \forall \delta<\frac12s_{gap},\\
        [\mathcal{R}(t)]_{C^\delta}&\les \eps_0^2(1+\eps) t^{-1-\delta}, \quad \forall \delta<\min\big(\frac13s_{gap},\frac18\big).
    \end{align*}
\end{lem}

Indeed, granting the lemma for now, using the $L^\infty$ estimate for $\mathcal{R}$ and integrating in time completes the proof of Proposition \ref{profilebd}:
\begin{align*}
    \|\hat\psi(t)\|_{L^\infty_\xi} = \|\hat{w}(t)\|_{L^\infty_\xi} &\leq \|\hat{w}(1)\|_{L^\infty_\xi} + \int_1^t \|\mathcal{R}(s)\|_{L^\infty_\xi} ds \\
    &\les \eps_0 + \eps_0^2(1+\eps) \les \eps_0(1+\eps).
\end{align*}

\end{proof}

\begin{proof}[Proof of Lemma \ref{remest}] Denoting by $\tilde{g}$ the Fourier transform of $g$ with respect to $v$, we compute 
\begin{align*}
    \mathcal{F}^{-1}_{v,\eta}[\hat{g}(t,\eta,\xi-v)\hat\psi(t,\xi-\eta)] &= \mathcal{F}^{-1}_\eta \big[ \hat\psi(t,\xi-\eta)\mathcal{F}^{-1}_v[\hat{g}(t,\eta,\xi-v)]  \big]\\
    &= (2\pi)^{-3} e^{i\xi\cdot v} \mathcal{F}^{-1}_\eta \big[ \hat\psi(t,\xi-\eta)\hat{\tilde{g}}(t,\eta,v)  \big]\\
    &= (2\pi)^{-3} e^{i\xi\cdot v} \mathcal{F}^{-1}_\eta [ \hat\psi(t,\xi-\eta)] *_\eta  \mathcal{F}^{-1}_\eta[\hat{\tilde{g}}(t,\eta,v)  ]\\
    &= (2\pi)^{-3} e^{i\xi\cdot (v+\eta)}\int e^{-iy\cdot\xi}\psi(t,y-\eta)\tilde{g}(t,y,v)dy.
\end{align*}

Therefore, by making the change of variables $\eta\mapsto z=v+\eta$, we get
\begin{align}\label{Rrep}
     \mathcal{R}(t,\xi) &= (2\pi)^3 t^{-1}\iint |v|^{-2}[e^{-i(|v|^2/2 + v\cdot\eta)/t} -1]\mathcal{F}^{-1}_{v,\eta}[\hat{g}(t,\eta,\xi-v)\hat\psi(t,\xi-\eta)]dv d\eta \notag \\
    &=  t^{-1}\iint |v|^{-2}[e^{-i(|v|^2/2+v\cdot\eta)/t} -1] e^{i\xi\cdot(v+\eta)} \int e^{-iy\cdot\xi} \psi(t,y-\eta)\tilde{g}(t,y,v)dy d\eta dv \notag \\
    &=  t^{-1}\iint e^{i\xi\cdot z} |v|^{-2}[e^{i(|v|^2/2-v\cdot z)/t} -1]  \int e^{-iy\cdot\xi}\psi(t,y-z+v)\tilde{g}(t,y,v)dy dv dz  \notag \\
    &= (2\pi)^3 t^{-1} \mathcal{F}^{-1}_{z\to\xi} \int |v|^{-2}[e^{i(|v|^2/2-v\cdot z)/t}-1]\int e^{-iy\cdot\xi} \psi(t,y-z+v)\tilde{g}(t,y,v)dydv.
\end{align}

Using Hausdorff-Young, $|e^{ix}-1|\leq 2|x|^{\delta}$, and Young's inequality, we obtain
\begin{align*}
    |\mathcal{R}(t,\xi)|
    &\les t^{-1-\delta}\iint |v|^{-2}\left(|z|^{2\delta} +|v|^{2\delta}\right) \int |\psi(t,y-z+v)||\Tilde{g}(t,y,v)|dy dz dv\\
    &\les t^{-1-\delta} \iiint |v|^{-2}(|y-z+v|^{2\delta} + |y|^{2\delta} + |v|^{2\delta})|\psi(t,y-z+v)||\Tilde{g}(t,y,v)|dzdydv\\
    &= t^{-1-\delta} \iiint |v|^{-2}(|z|^{2\delta}+|y|^{2\delta}+|v|^{2\delta})|\psi(s,z)||\Tilde{g}(t,y,v)|dzdydv\\
    &=: \mathcal{R}_1(t,\xi)+\mathcal{R}_2(t,\xi)+\mathcal{R}_3(t,\xi).
\end{align*}

Next we prove two estimates used often in estimating the terms above. Using the pointwise bound $|\Tilde{g}(t,y,v)| \leq \int g(t,y,u)du$, Cauchy-Schwartz, Plancherel in $v$, and Lemma \ref{vlasovweighted}, for any $3/2<p<3$ and $0\leq q<1/2$ we have
\begin{align}\label{vxweightedest}
    \iint |v|^{-p} |y|^q |\tilde{g}(t,y,v)|dv dy &= \left(\iint_{|v|\leq 1} + \iint_{|v|>1} \right)|v|^{-p}|y|^q |\Tilde{g}(t,y,v)|dv dy\notag\\
    &\leq \big\||v|^{-p} \big\|_{L^1_v(B_1)}\|\jbrac{x}^q g\|_{L^1_{x,v}} + \big\||v|^{-p}\big\|_{L^2_v(B_1^c)}\int |y|^q \|\Tilde{g}(t,y)\|_{L^2_v} dy\notag\\
    &\les \|\jbrac{x}^q g\|_{L^1_{x,v}} + \big\|\jbrac{x}^{-2+q} \big\|_{L^2_x}\|\jbrac{x}^2g\|_{L^2_{x,v}}\notag\\
    &\les \eps_0\ln^4\jbrac{t}.
\end{align}

Similarly, using H\"older and Proposition \ref{EEhartree}, as long as $0\leq q<s_{gap}$ we get
\begin{equation}\label{lowmomentest}
    \||z|^q\psi\|_{L^1_z} \les \|\jbrac{z}^{-s+q}\|_{L^2_z}\|\jbrac{z}^s\psi\|_{L^2_z} \les \|\phi\|_{L^2} + \||J|^s\phi\|_{L^2} \les \eps_0(1+\eps)\ln^2\jbrac{t}.
\end{equation}

Finally, we are ready to estimate the remainder. Since $2\delta<\min(s_{gap},\frac12)= s_{gap}$, we get
\begin{align*}
    \mathcal{R}_1(t,\xi) &\les t^{-1-\delta}  \||z|^{2\delta}\psi(t,z)\|_{L^1_z}\iint |v|^{-2}|\Tilde{g}(t,y,v)|dydv  \les \eps_0^2(1+\eps) t^{-1-\delta}\ln^6\jbrac{t},\\
    \mathcal{R}_2(t,\xi) &\les t^{-1-\delta} \|\psi(t,z)\|_{L^1_z}\iint |v|^{-2}|y|^{2\delta}|\tilde{g}(t,y,v)|dydv \les \eps_0^2(1+\eps) t^{-1-\delta}\ln^6\jbrac{t},\\
    \mathcal{R}_3(t,\xi) &\les t^{-1-\delta} \|\psi(t,z)\|_{L^1_z} \iint |v|^{-2+2\delta} |\tilde{g}(t,y,v)|dydv \les \eps_0^2(1+\eps) t^{-1-\delta}\ln^6\jbrac{t}.
\end{align*}

This almost gives the $L^\infty$ estimate, but we have additional logarithmic growth factors. These can easily be absorbed, possibly with a larger implicit constant, by noting that the condition $\delta<\frac12 s_{gap}$ is an open condition and any positive power of $\jbrac{t}$ grows faster than any power of $\ln\jbrac{t}$. We will use this idea frequently throughout the rest of the paper to absorb logarithmic factors as they appear.\\

To obtain $L^2$ estimates, we apply Plancherel and Minkowski's inequality to (\ref{Rrep}), and the same estimates used to bound $\mathcal{R}$ in $L^\infty$ yield 
\begin{align*}
    \|\mathcal{R}(t,\xi)\|_{L^2_\xi} &= t^{-1} \left\|\int |v|^{-2}[e^{i(|v|^2/2-v\cdot z)/t}-1]\int \psi(t,y-z+v)\tilde{g}(t,y,v)dydv \right\|_{L^2_z}\\
    &\les t^{-1-\delta}\left\|\int |v|^{-2}(|v|^{2\delta}+|z|^{2\delta}) \int \psi(t,y-z+v)\tilde{g}(t,y,v)dydv \right\|_{L^2_z}\\
    &\les t^{-1-\delta}\left\|\iint |v|^{-2}(|v|^{2\delta}+|z-v-y|^{2\delta} + |y|^{2\delta}) \psi(t,y-z+v)\tilde{g}(s,y,v)dydv \right\|_{L^2_z}\\
    &\les t^{-1-\delta}\Big(\iint |v|^{-2+2\delta}|\tilde{g}(t,y,v)| \big\| \psi(t,y-z+v)\big\|_{L^2_z} \\
    &\qquad \qquad \qquad + |v|^{-2}|\tilde{g}(t,y,v)| \big\||y-z+v|^{2\delta} \psi(t,y-z+v)\big\|_{L^2_z}\\
    &\qquad \qquad \qquad + |v|^{-2}|y|^{2\delta}|\tilde{g}(t,y,v)| \big\| \psi(t,y-z+v)\big\|_{L^2_z} dy dv\Big)\\
    &\les \eps_0^2(1+\eps) t^{-1-\delta}.
\end{align*}

In the above we used (\ref{vxweightedest}), $\||\cdot|^{\delta}\psi\|_{L^2}\les \|\phi\|_{L^2}+\||J|^s\phi\|_{L^2}\les \eps_0(1+\eps)\ln^2\jbrac{t}$, and the same idea just mentioned to absorb the logarithmic factors coming from the moment estimates. \\

To estimate the H\"older seminorm of $\mathcal{R}$ we will instead bound $|\nabla_\xi|^{2\delta}\mathcal{R}$ in $L^{3/\delta}$ and then apply Morrey's inequality $\dot{H}^{r,p}\hookrightarrow C^{\alpha}$ where $\alpha=r-\frac{3}{p}$ and $\dot{H}^{r,p}$ is the Riesz potential space $\dot{H}^{r,p}=\{u: u=|\cdot|^{r-3}*v, \ v\in L^p\}$. Using (\ref{Rrep}) and the same ideas as the $L^2$ estimate but with Hausdorff-Young and Young's inequality $|v|^{\delta}|z|^{3\delta} \les |v|^{4\delta}+|z|^{4\delta}$ gives
\begin{align*}
    \||\nabla_\xi|^{2\delta}\mathcal{R}(t,\xi)\|_{L^{3/\delta}_\xi} &\les t^{-1} \left\||z|^{2\delta}\int |v|^{-2}[e^{i(|v|^2/2-v\cdot z)/t}-1]\int \psi(t,y-z+v)\tilde{g}(t,y,v)dydv \right\|_{L^{\frac{3}{3-\delta}}_z}\\
    &\les t^{-1-\delta}\left\|\int |v|^{-2}(|v|^{4\delta}+|z|^{4\delta}) \int \psi(t,y-z+v)\tilde{g}(t,y,v)dydv \right\|_{L^{\frac{3}{3-\delta}}_z}\\
    &\les t^{-1-\delta}\left\|\iint |v|^{-2}(|v|^{4\delta}+|z-v-y|^{4\delta} + |y|^{4\delta}) \psi(t,y-z+v)\tilde{g}(t,y,v)dydv \right\|_{L^{\frac{3}{3-\delta}}_z}\\
    &\les t^{-1-\delta}\Big(\iint |v|^{-2+4\delta}|\tilde{g}(t,y,v)| \big\| \psi(t,y-z+v)\big\|_{L^{\frac{3}{3-\delta}}_z} \\
    &\qquad \qquad \qquad + |v|^{-2}|\tilde{g}(t,y,v)| \big\||y-z+v|^{4\delta} \psi(t,y-z+v)\big\|_{L^{\frac{3}{3-\delta}}_z}\\
    &\qquad \qquad \qquad + |v|^{-2}|y|^{4\delta}|\tilde{g}(t,y,v)| \big\| \psi(t,y-z+v)\big\|_{L^{\frac{3}{3-\delta}}_z} dy dv\Big).
\end{align*}

We have
\begin{align*}
    \||x|^{\beta}\psi\|_{L^{\frac{3}{3-\delta}}}\les \|\jbrac{x}^{-s+\beta}\|_{L^{p_1}}\|\jbrac{x}^s\psi\|_{L^2} \les \eps_0(1+\eps)\ln^2\jbrac{t},
\end{align*}
which holds for all $\beta\geq 0$ satisfying
$$s-\beta>\frac{3}{p_1} = 3\Big( -\frac12 + \frac{3-\delta}{3} \Big) = \frac32 - \delta \implies \beta < s_{gap}+\delta.$$
Since $4\delta<s_{gap}+\delta$ in view of $\delta<\frac13 s_{gap}$, this leads to
\begin{align*}
    \||x|^{4\delta} \psi\|_{L^{\frac{3}{3-\delta}}} + \|\psi\|_{L^{\frac{3}{3-\delta}}} \les \eps_0(1+\eps) \ln^2\jbrac{t}.
\end{align*}
Next, recalling (\ref{vxweightedest}) and the fact that $4\delta < \frac12$, we obtain
\begin{align*}
    \iint (|v|^{-2+4\delta} +|v|^{-2}|y|^{4\delta})|\tilde{g}(t,y,v)|dydv \les \eps_0\ln^4\jbrac{t}.
\end{align*}

Therefore, Morrey's inequality yields
$$[\mathcal{R}(t)]_{C^{\delta}}\les \||\nabla_\xi|^{2\delta} \mathcal{R}(t,\xi)\|_{L^{3/\delta}_{\xi}} \les \eps_0^2(1+\eps) t^{-1-\delta}\ln^6\jbrac{t}.$$
Again, we absorb the log factors by appealing to the fact that $\delta<\min(\frac13s_{gap},\frac18)$ is an open condition.

\end{proof}

\section{Asymptotic behavior}\label{asympbehavior}
We now turn to proving Theorem \ref{modscat}. Since Theorem \ref{decay} now holds, any bounds of the form $(\cdots)\les \eps_0(1+\eps)$ or $(\cdots)\les \eps$ from previous sections will be replaced by $(\cdots)\les \eps_0$; recall Proposition \ref{bootstrap} and the arguments following it.

\subsection{Construction of the phase correction for the bosons}

In this section we construct the limiting profile $\phi_\infty$ and phase $\Psi_\infty$ and prove the asymptotic formula (\ref{phiasymp}) for $\phi$. We use techniques inspired from \cite{hayashi_asymptotics_1998}.\\

Since the remainder $\mathcal{R}$ decays at the integrable rate $t^{-1-\delta}$ in $L^2\cap L^\infty$ by Lemma \ref{remest}, we may simply integrate to obtain the limiting profile $\phi_\infty$:

$$\phi_\infty(\xi) = \lim_{t\to\infty} \hat{w}(t,\xi) = \hat{w}(1,\xi) -i \int_1^\infty B(t,\xi) \mathcal{R}(t,\xi) dt \quad \text{in } L^\infty_{\xi}\cap L^2_\xi.$$

As a consequence, it follows that
$$\|\phi_\infty-\hat{w}(t)\|_{L^2_\xi \cap L^\infty_{\xi}} \les \eps_0^2 t^{-\delta}.$$ 

Next we prove convergence of $\hat{w}(t)$ in $C^\delta$ for $\delta<\min(\frac13 s_{gap},\frac18)$. We have
\begin{align*}
    [\phi_\infty-\hat{w}(t)]_{C^\delta}\leq  \int_t^\infty \big([B(t)]_{C^\delta} \|\mathcal{R}(t)\|_{L^\infty_\xi} + \|B(t)\|_{L^\infty_\xi}[\mathcal{R}(t)]_{C^\delta}\big) dt.
\end{align*}

Note that by Lemmas \ref{potest} and \ref{gcharbds},
\begin{align*}
    [B(t)]_{C^\delta}\les [P(t)]_{C^\delta} &\leq \int_1^t s^{-1}[V*m(s)]_{C^\delta} ds \les \int_1^t s^{-1}\||\cdot|^{-1-\delta}*m(s)\|_{L^\infty}ds\\
    &\les\int_1^t s^{-1}\|m(s)\|_{L^1\cap L^\infty} ds \les \eps_0 \ln^4\jbrac{t}.
\end{align*}

Hence, using Lemma \ref{remest} and using the open condition to absorb the logarithmic factors, we get
\begin{align*}
    [\phi_\infty-\hat{w}(t)]_{C^\delta} \les \int_t^\infty \eps_0^2 t^{-1-\delta}\ln^4\jbrac{t}dt \les \eps_0^2 t^{-\delta}.
\end{align*}
Hence, $\phi_\infty\in C^\delta$ since by Morrey's inequality and Proposition \ref{EEhartree}, $$[\hat{w}(1)]_{C^\delta} = [\hat\psi(1)]_{C^\delta}\les \|\hat\psi(1)\|_{H^s} \les \|\phi(1)\|_{L^2} + \||J|^s\psi(1)\|_{L^2} < \infty.$$

Next, we use the convergence of the scattering mass to derive the asymptotic phase correction in terms of $V*m_\infty$. Indeed, let
\begin{align*}
    \Psi(t,\xi) &:= P(t,\xi) - \ln(t)(V* m)(t,\xi) = \int_1^t \frac1s [(V*m)(s,\xi) - (V* m)(t,\xi)] ds,  \\
    \Psi_\infty(\xi) &:= \lim_{t\to\infty}\Psi(t,\xi) = \int_1^\infty \frac1s [(V*m)(s,\xi) - (V* m_\infty)(\xi)] ds .
\end{align*}

From Lemma \ref{scatmass}, $\Psi_\infty$ is well-defined. Moreover, a direct calculation yields 
\begin{align*}
    P(t,\xi)-(\Psi_\infty(\xi)+\ln(t)(V*m_\infty)) = \int_t^\infty \frac1s [(V*m_\infty)(\xi)-(V*m)(s,\xi)]ds.
\end{align*}
Therefore, Lemma \ref{scatmass} yields
$$\|P(t,\xi)-(\Psi_\infty(\xi)+\ln(t)(V*m_\infty))\|_{L^\infty_\xi}\les \eps_0 \jbrac{t}^{-1}$$

Finally, for the claimed asymptotic formula (\ref{phiasymp}), we write
\begin{align*}
    \phi(t,x) - \frac{e^{i|x|^2/2t}}{(it)^{3/2}}&\phi_\infty(\tfrac{x}{t}) \exp\big( -i(\Psi_\infty(\tfrac{x}{t}) + \ln(t) (V*m_\infty)(\tfrac{x}{t})) \big) \\
    &= R_1(t,x)+R_2(t,x)+R_3(t,x),
\end{align*}
where
\begin{align*}
    R_1(t,x) &= \phi(t,x)-\frac{e^{i|x|^2/2t}}{(it)^{3/2}} \hat{\psi}(t,\tfrac{x}{t}),\\
    R_2(t,x) &= \frac{e^{i|x|^2/2t}}{(it)^{3/2}}e^{-iP(t,\tfrac{x}{t})}\big( \hat{w}(t,\tfrac{x}{t}) - \phi_\infty(\tfrac{x}{t}) \big),\\
    R_3(t,x) &= \frac{e^{i|x|^2/2t}}{(it)^{3/2}} \phi_\infty(\tfrac{x}{t}) \big( e^{-iP(t,\tfrac{x}{t})} - e^{-i\big(\Psi_\infty(\tfrac{x}{t}) + \ln(t) (V*m_\infty)(\tfrac{x}{t})\big)} \big).
\end{align*}
First, $R_1$ is the remainder in Lemma \ref{fraunhofer} and satisfies, by Proposition \ref{EEhartree},
\begin{align*}
    \|R_1(t)\|_{L^\infty_x} &\les t^{-3/2-\delta} \|\jbrac{x}^s\psi(t)\|_{L^2_x} \les \eps_0 t^{-3/2-\delta} \ln^2\jbrac{t},\\
    \|R_1(t)\|_{L^2_x} &\les t^{-s/2} \|\jbrac{x}^s\psi(t)\|_{L^2_x} \les \eps_0 t^{-3/4}\ln^2\jbrac{t}.
\end{align*}

Similar to Lemma \ref{remest}, the $\ln^2\jbrac{t}$ factor appearing in the first estimate can be absorbed by slightly increasing $\delta$ in the proof of Lemma \ref{fraunhofer}, since the condition $\delta<\frac12 s_{gap}$ is open. The bound on $R_2$ follows from the convergence of $\hat{w}$ to $\phi_\infty$: 
\begin{align*}
    \|R_2(t)\|_{L^\infty_x} &\les t^{-3/2}\|\hat{w}(t)-\phi_\infty\|_{L^\infty} \les \eps_0^2 t^{-3/2-\delta},\\
    \|R_2(t)\|_{L^2_x} &\les t^{-3/2}\|\hat{w}(t,\tfrac{x}{t})-\phi_\infty(\tfrac{x}{t})\|_{L^2_x} = \|\hat{w}(t,y)-\phi_\infty(y)\|_{L^2_y} \les \eps_0^2 t^{-\delta}.
\end{align*}

Finally, the bound on $R_3$ is a consequence of the convergence of $P(t)$ and the bound $|e^{ix}-e^{iy}|\leq 2|x-y|$:
\begin{align*}
    \|R_3(t)\|_{L^\infty_x} &\les t^{-3/2}\|P(t) - (\Psi_\infty + \ln(t)(V*m_\infty)) \|_{L^\infty} \|\phi_\infty\|_{L^\infty} \les \eps_0^4 t^{-3/2-1}\ln(t),\\
    \|R_3(t)\|_{L^2_x} &\les t^{-3/2}\|P(t) - (\Psi_\infty + \ln(t)(V*m_\infty)) \|_{L^\infty} \|\phi_\infty(\tfrac{x}{t})\|_{L^2_x} \les \eps_0^4 t^{-1}\ln(t).
\end{align*}

\subsection{Asymptotic force generated by the bosons}

In this section, we identify and prove convergence to the asymptotic dynamics of the field $E$ along the trajectories of free transport. This will be the key estimate enabling us to prove modified scattering for the fermionic subsystem.\\

Let $|\alpha|\leq 1$ be an arbitrary multi-index. By factoring out the Schr\"odinger free dynamics and using $\hat{\bar\phi}(\xi) = \overline{\hat\phi(-\xi)}$,
\begin{align*}
    (\p^\alpha_x E)(t,x+vt) &= -(2\pi)^{-3}\iint e^{i(x+vt)\cdot\xi}  (i\xi)^{1+\alpha}\hat{V}(\xi) \hat{\phi}(t,\eta)\overline{\hat\phi(t,\eta-\xi)}d\eta d\xi\\
    &= -(2\pi)^{-3}\iint e^{i(x+vt)\cdot\xi}(i\xi)^{1+\alpha}\hat{V}(\xi) e^{it|\xi|^2/2} e^{-it\xi\cdot\eta}  \hat{\psi}(t,\eta)\overline{\hat\psi(t,\eta-\xi)}d\eta d\xi\\
    &= -\int e^{i(x+vt)\cdot\xi} (i\xi)^{1+\alpha}\hat{V}(\xi) e^{it|\xi|^2/2}   \mathcal{F}^{-1}_{\eta}[\hat{\psi}(t,\eta)\overline{\hat\psi(t,\eta-\xi)}](-t\xi) d\xi.
\end{align*}

We compute
\begin{align*}
    \mathcal{F}^{-1}_{\eta}[\hat{\psi}(t,\eta)\overline{\hat\psi(t,\eta-\xi)}](-t\xi) &= [\mathcal{F}^{-1}_{\eta}(\hat\psi(t,\eta)) * \mathcal{F}^{-1}_{\eta}(\hat{\bar{\psi}}(t,\xi-\eta))](-t\xi)\\
    &= \int e^{i(-t\xi-y)\cdot\xi} \psi(t,y)\overline{\psi(y+t\xi)} dy.
\end{align*}

Inserting this back into the expression for $(\p^\alpha_xE)(t,x+vt)$ and changing variables $\sigma = -t\xi$ gives
\begin{align*}
    (\p^\alpha_x E)(t,x+vt) &= -\iint e^{i((x-y)+vt)\cdot\xi}e^{-it|\xi|^2/2} (i\xi)^{1+\alpha}\hat{V}(\xi) \psi(t,y)\overline{\psi(t,y+t\xi) }dy d\xi\\
    &= t^{-2-|\alpha|}\iint e^{-i(x-y)\cdot\sigma/t}e^{-i|\sigma|^2/2t} e^{-iv\cdot\sigma} (i\sigma)^{1+\alpha}\hat{V}(\sigma)  \psi(t,y)\overline{\psi(t,y-\sigma)} dy d\sigma.
\end{align*}

In the $t\to\infty$ limit, the factors with phase $O(1/t)$ converge to $1$, so to extract the leading order term we replace those factors by $1$ and simplify:
\begin{align*}
    \text{Leading order term} &= t^{-2-|\alpha|}\iint e^{-iv\cdot\sigma} (i\sigma)^{1+\alpha}\hat{V}(\sigma) \psi(t,y)\overline{\psi(t,y-\sigma)} dyd\sigma \\
    &= t^{-2-|\alpha|}\mathcal{F}_{\sigma}\big\{\widehat{\p^\alpha\nabla V}(\sigma) [\psi(t,y)*\bar\psi(t,-y)](\sigma)\big\}(v)\\
    &= t^{-2-|\alpha|} [(\p^\alpha\nabla V)(-u) * (\hat\psi(t,u)\bar{\hat{\psi}}(t,u))](v)\\
    &= -t^{-2-|\alpha|} [\p^\alpha\nabla V * |\hat\psi(t)|^2](v).
\end{align*}
Furthermore, since $\hat{w}(t,\xi) = B(t,\xi)\hat{\psi}(t,\xi)$ converges to a modified asymptotic profile $\phi_\infty(\xi)$ as $t\to\infty$, and $|B(t,\xi)|=1$, we expect $-\nabla V *|\hat\psi(t)|^2$ to be well-approximated by the limiting field profile 
$$E_\infty(v) := -(\nabla V * |\phi_\infty|^2)(v).$$
Putting everything together, we write
$$(\p^\alpha_x E)(t,x+vt) = t^{-2-|\alpha|}\p^\alpha_v E_\infty(v) + t^{-2-|\alpha|}\text{Err}^\alpha(t,x,v),$$
where
\begin{align*}
    \text{Err}^\alpha(t,x,v) &= \text{Err}^\alpha_1(t,x,v)+\text{Err}^\alpha_2(t,x,v) + \text{Err}^\alpha_3(t,v),\\
    \text{Err}^\alpha_1(t,x,v) &= \iint (e^{-i(x-y)\cdot\sigma/t}-1) e^{-iv\cdot\sigma} (i\sigma)^{1+\alpha}\hat{V}(\sigma)  \psi(t,y)\overline{\psi(t,y-\sigma)} dy d\sigma,\\
    \text{Err}^\alpha_2(t,x,v) &= \iint e^{-i(x-y)\cdot\sigma/t}(e^{-i|\sigma|^2/2t}-1)  e^{-iv\cdot\sigma} (i\sigma)^{1+\alpha}\hat{V}(\sigma) \psi(t,y)\overline{\psi(t,y-\sigma)} dy d\sigma,\\
    \text{Err}^\alpha_3(t,v) &= -\big[\p^\alpha\nabla V * (|\hat\psi|^2-|\phi_\infty|^2)\big](v).
\end{align*}

The error is estimated in the following proposition.

\begin{prop}\label{Eerrest}
    For any $\delta<\min(\frac13 s_{gap},\frac18)$, the function $E_\infty$ is in $C^{1,\delta}$ with $\|E_\infty\|_{C^{1,\delta}}\les \eps_0^2$ and 
\begin{align*}
    \|E(t,x+vt)-t^{-2}E_\infty(v)\|_{L^\infty_v}&\les \eps_0^2 \jbrac{x}^{2\delta} t^{-2-\delta},\\
    \|(\nabla_xE)(t,x+vt)-t^{-3}\nabla E_\infty(v)\|_{L^\infty_v}&\les \eps_0^2 \jbrac{x}^{2\delta} t^{-3-\delta}.
\end{align*}
\end{prop}

\begin{proof}
    The fact that $E_\infty\in C^{1,\delta}$ follows immediately from elliptic regularity and the fact that $\phi_\infty\in C^\delta\cap L^2\cap L^\infty$.\\
    
    To prove convergence, we start with $\text{Err}^\alpha_3$ since it is the easiest. Lemma \ref{potest} and the convergence of $\hat{w}$ to $\phi_\infty$ in $L^2\cap L^\infty$ gives
\begin{align*}
    |\text{Err}^0_3(t,v)| &\leq |\nabla V * (|\phi_\infty|^2-|\hat\psi(t)|^2)|\\
    &\les \||\phi_\infty|^2-|\hat\psi(t)|^2\|_{L^1}^{1/3}\||\phi_\infty|^2-|\hat\psi(t)|^2\|_{L^\infty}^{2/3}\\
    &\les \|\phi_\infty+\hat{w}(t)\|_{L^2\cap L^\infty}\|\phi_\infty-\hat{w}(t)\|_{L^2}^{1/3}\|\phi_\infty-\hat{w}(t)\|_{L^\infty}^{2/3}\\
    &\les \eps_0^3 t^{-\delta}.
\end{align*}

To handle $|\alpha|=1$, first note that standard elliptic estimates give
\begin{equation}\label{holderest}
        \|\nabla^2 V*u\|_{L^\infty}\les \|u\|_{C^\delta} + \|u\|_{L^1}.
    \end{equation}

so that the convergence of $\hat{w}$ to $\phi_\infty$ in $L^2\cap C^\delta$ yields
\begin{align*}
    |\text{Err}^\alpha_3(t,v)| &\les \||\phi_\infty|^2-|\hat\psi(t)|^2\|_{C^\delta}+\||\phi_\infty|^2-|\hat\psi(t)|^2\|_{L^1}\\
    &\les \|\phi_\infty+\hat{w}(t)\|_{C^\delta}\|\phi_\infty-\hat{w}(t)\|_{C^\delta}+\|\phi_\infty+\hat{w}(t)\|_{L^2}\|\phi_\infty-\hat{w}(t)\|_{L^2}\\
    &\les \eps_0^3 t^{-\delta}.
\end{align*}

To estimate the other two terms, first note that as long as $p\in(-s_{gap},3)$ we have
\begin{align*}
    \int |\sigma|^{-p} |\psi(t,y-\sigma)|d\sigma &= \left(\int_{|\sigma|\leq 1} + \int_{|\sigma|>1}\right) |\sigma|^{-p} |\psi(t,y-\sigma)|d\sigma\\
    &\les \||\cdot|^{-p}\|_{L^1(B_1)} \|\psi(t)\|_{L^\infty} + \|\jbrac{\cdot}^{-p-s}\|_{L^2(B_1^c)}\||x|^s\psi(t,x)\|_{L^2}\\
    &\les \|\phi\|_{H^s} + \|J^s\phi\|_{L^2}\\
    &\les \eps_0\ln^2\jbrac{t}.
\end{align*}
Therefore, using $|e^{ix}-1|\leq 2|x|^{2\delta}$ and recalling (\ref{lowmomentest}),
\begin{align*}
    |\text{Err}^\alpha_1(t,x,v)| &\les t^{-2\delta}\iint |x-y|^{2\delta} |\sigma|^{-1+|\alpha|+2\delta} |\psi(t,y)||\psi(t,y-\sigma)|dyd\sigma\\
    &\les \eps_0 t^{-2\delta}  \ln^2\jbrac{t} \big(|x|^{2\delta} \|\psi(t,y)\|_{L^1_y} + \||y|^{2\delta} \psi(t,y)\|_{L^1_y} \big)\\
    &\les \eps_0^2 \jbrac{x}^{2\delta} t^{-\delta}\\
    |\text{Err}^\alpha_2(t,x,v)| &\les t^{-2\delta}\iint |\sigma|^{-1+|\alpha|+2\delta} |\psi(t,y)||\psi(t,y-\sigma)|dyd\sigma\\
    &\les \eps_0 t^{-2\delta} \ln^2\jbrac{t}\|\psi(t)\|_{L^1}\\
    &\les \eps_0^2 t^{-\delta}.
\end{align*}

Also $1-|\alpha|-2\delta\in (-s_{gap},3)$ since $|\alpha|\leq 1$ and $2\delta < \frac23 s_{gap}<s_{gap}$.
\end{proof}

In a similar vein to the previous result and Proposition \ref{asympfermpot}, we also have the following self-similar asymptotic expression for the force $E$ in the rest frame, which proves (\ref{Easymp}).

\begin{prop}\label{asympbosonforce}
    There holds 
    $$\|E(t,x)-t^{-2}E_\infty(\tfrac{x}{t})\|_{L^\infty_x} \les \eps_0^2 t^{-2-\delta}$$
\end{prop}
\begin{proof}
    First, note that the change of variables $z=vt$  and the $-2$ homogeneity of $\nabla V$ yields
$$t^{-2}E_\infty\big( \tfrac{x}{t}) = -t^{-2}\int \nabla V(v) |\phi_\infty\big(\tfrac{x}{t}-v\big)|^2 dv = -t^{-3} \int \nabla V(z) |\phi_\infty\big(\tfrac{x-z}{t}\big)|^2 dz.$$
Then Lemma \ref{potest} yields
    \begin{align*}
        \|E(t,x)-t^{-2}E_\infty(\tfrac{x}{t})\|_{L^\infty_x} \les \big\||\phi(t,x)|^2-t^{-3}|\phi_\infty(\tfrac{x}{t})|^2\big\|_{L^1_x}^{1/3}\big\||\phi(t,x)|^2-t^{-3}|\phi_\infty(\tfrac{x}{t})|^2\big\|_{L^\infty_x}^{2/3}.
    \end{align*}
    Also,
    \begin{align*}
        |\phi(t,x)&|^2-t^{-3}|\phi_\infty(\tfrac{x}{t})|^2\\ 
        &\leq \big(|\phi(t,x)|+t^{-3/2}|\phi_\infty(\tfrac{x}{t})|\big) \Big|\phi(t,x)- \frac{e^{i|x|^2/2t}}{(it)^{3/2}} \phi_\infty(\tfrac{x}{t}) \exp\big(-i\Psi_\infty(\tfrac{x}{t}) - i\ln(t)(V*m_\infty)(\tfrac{x}{t}) \big)\Big|.
    \end{align*}
    Hence, using the convergence of the asymptotic formula for $\phi$ in $L^2_x\cap L^\infty_x$, we get
    {\small\begin{align*}
        \|E(t,x)-t^{-2}E_\infty(\tfrac{x}{t})\|_{L^\infty_x} \les \big( (\|\phi\|_{L^2} + \|\phi_\infty\|_{L^2})t^{-\delta}\big)^{1/3}\big( (\|\phi\|_{L^\infty} + t^{-3/2}\|\phi_\infty\|_{L^\infty})t^{-3/2-\delta}\big)^{2/3} \les \eps_0^2 t^{-2-\delta}.
    \end{align*}}
\end{proof}

\subsection{Asymptotic spatial characteristics}\label{sec:xlim}
In this section we construct the asymptotic spatial characteristics of the fermions. Since the fermions disperse, the asymptotic behavior of the particle positions can only be meaningfully defined in a moving frame. However, long-range effects cause deviations from the frame associated to free transport, the obvious candidate for a frame in which to establish a limit. To derive the modified frame, consider the following heuristics. We know that in the frame of the free flow $(x,v)\mapsto (x+vt,v)$, the force $E$ generated by the bosons can be well-approximated by $t^{-2}E_\infty(v)$ for large times by Proposition \ref{Eerrest}, and moreover by Proposition \ref{Vlim} the velocity limit is given by $V_\infty = v+O(\eps_0^2)$, so the force acting on a particle $(x,v)$ for $t$ very large is roughly $t^{-2}E_\infty(v)$, which means the asymptotic position of said particle is roughly $x+vt-\ln(t)E_\infty(v)$ by Newton's laws. This is the frame in which we establish a limit. To this end, define the modified forward free flow and its inverse:
\begin{align*}
    \mathcal{M}_t(x,v)&=(x+vt-\ln(t)E_\infty(v),v),\\
    \mathcal{M}_t^{-1}(x,v)&=(x-vt+\ln(t)E_\infty(v),v).
\end{align*}
The object which we expect to have a limit as $t\to\infty$ is
\begin{equation*}
    \mathcal{A}_{t,s} := \mathcal{M}_t^{-1} \circ \Phi_{t,s},
\end{equation*}
where recall that $\Phi_{t,s}$ is the \emph{forward} ($t\geq s$) nonlinear flow. The limit $\mathcal{A}_{\infty,s}(x,v)$ which we will construct later in this section describes the asymptotic coordinates \emph{in the frame of the modified free flow} of a particle initially at $(x,v)$ at time $s$. In particular this means the forward particle trajectories may be approximated by $\Phi_{t,s}\approx \mathcal{M}_t\circ \mathcal{A}_{\infty,s}$ for large $t$. We write 
$$\mathcal{A}_{t,s}(x,v) := (Z(t,s,x,v),V(t,s,x,v)).$$ 

Before constructing the spatial limit, we need some preliminary estimates on 
$$Y(t,s,x,v):= X(t,s,x,v)-tV(t,s,x,v).$$

\begin{lem}\label{Ybds}
    Let $U_0\subset\R^6$ be a bounded open set, and let $U_s=\Phi_{s,0}(U_0)$. Then for any $t\geq s$ there holds
    \begin{align*}
        |Y(t,s,x,v)|&\les_{U_0} \begin{cases}
            \ln\jbrac{t} & (x,v)\in U_s,\\
            \jbrac{s}+\eps_0^2\ln\jbrac{t} & (x,v)\in U_0,
        \end{cases}\\
        |\nabla_x Y(t,s,x,v)|&\les 1,\\
        |\nabla_v Y(t,s,x,v)|&\les s+\eps_0^2\ln^2\jbrac{t}.
    \end{align*}
\end{lem}
\begin{proof}
    First, the characteristic ODEs yield
    $$Y(t,s,x,v) = x-vs - \int_s^t \tau E(\tau,X(\tau,s,x,v))d\tau.$$
    Next, an easy modification of Lemma \ref{gcharbds} shows that that $\sup\{|x-vs| : (x,v)\in \Phi_{s,0}(U_0)\}\les \ln\jbrac{s}$, by replacing $\mathcal{S}_0$ with $U_0$. Using the decay of $E$ and the fact that $t\geq s$, this implies that
    \begin{align*}
        |Y(t,s,x,v)| \leq |x-vs| + \int_0^t \tau |E(\tau,X(\tau,t,x,v))| d\tau \les_{U_0} \begin{cases}
             \ln\jbrac{t} & (x,v)\in U_s,\\
            \jbrac{s}+\eps_0^2\ln\jbrac{t} & (x,v)\in U_0.
        \end{cases}
    \end{align*}
    For the derivatives, we have
    \begin{align*}
        \nabla_x Y(t) = I-\int_s^t \tau\nabla_x E(\tau,X(\tau))\nabla_x X(\tau)d\tau,\\
        \nabla_v Y(t) = -sI-\int_s^t \tau\nabla_x E(\tau,X(\tau))\nabla_vX(\tau)d\tau.
    \end{align*}
    Therefore, Lemma \ref{forflowdbds} and the decay of $\nabla_x E$ yields
    \begin{align*}
        |\nabla_x Y(t)|&\les 1 + \int_s^t\eps_0^2 \jbrac{\tau}^{-2}\ln\jbrac\tau \ d\tau \les 1,\\
        |\nabla_v Y(t)|&\les s + \int_s^t \eps_0^2(\tau-s)\jbrac{\tau}^{-2}\ln\jbrac\tau \ d\tau \les s+\eps_0^2\ln^2\jbrac{t}.
    \end{align*}
\end{proof}

Now we can construct the spatial limit.

\begin{prop}\label{Xlim}
    Let $U_0$ be a bounded open set, and let $U_s := \Phi_{s,0}(U_0)$. The function 
    $$Z_\infty(s,x,v) := \lim_{t\to\infty} Z(t,s,x,v) = x-vs + \ln(s)E_\infty(v) + \int_s^\infty \frac{d}{d\tau}Z(\tau,s,x,v)d\tau.$$
    
    is a well-defined $C^1_{x,v}$ function, and the following estimates hold for any $s\geq 1$, with implicit constants which depend on the choice of $U_0$.\\
    
    (i) For any $\delta<\min(\frac13 s_{gap},\frac18)$, there holds
    \begin{align*}
        |Z_\infty(s,x,v)-Z(t,s,x,v)|&\les \begin{cases}
            \eps_0^2 t^{-\delta} & (x,v)\in U_s,\\
            \eps_0^2 t^{-\delta}\jbrac{s}^{2\delta} & (x,v)\in U_0,
        \end{cases}\\
        |\nabla_xZ_\infty(s,x,v)-\nabla_x Z(t,s,x,v)|&\les \begin{cases}
            \eps_0^2 t^{-\delta} & (x,v)\in U_s,\\
            \eps_0^2 t^{-\delta}\jbrac{s}^{2\delta} & (x,v)\in U_0,
        \end{cases}\\
        |\nabla_v Z_\infty(s,x,v)-\nabla_v Z(t,s,x,v)|&\les \eps_0^2 t^{-\delta}\jbrac{s}, \quad (x,v)\in U_s\cup U_0
    \end{align*}
    (ii) In particular, $Z(t,s)\to Z_\infty(s)$ as $t\to\infty$ locally in $C^1_{x,v}$ for every fixed $s\geq 1$, and 
    $$Z_{\infty}(s,x,v)=x-vs+\ln(s)E_\infty(v)+o(s) \text{ on } U_s \text{ as } s\to\infty.$$
    (iii) For any $s\geq 1$, the map $\mathcal{A}_{\infty,s}: U_0\to \mathcal{A}_{\infty,s}(U_0)$ defined by $$\mathcal{A}_{\infty,s}(x,v) := (Z_\infty(s,x,v),V_\infty(s,x,v))$$ is a $C^1_{x,v}$ volume-preserving diffeomorphism.
\end{prop}

\begin{remark}
    We choose to distinguish the estimates on $U_s$ versus $U_0$ because in practice one is primarily interested in the case when $(x,v)$ are in the support of $f(s)$ (i.e. $\mathcal{S}_s = \Phi_{s,0}(\mathcal{S}_0)$), since $f(t)\circ \mathcal{M}_t = f(s)\circ \mathcal{A}_{t,s}^{-1}$, and because the estimates on $U_s$ give asymptotics for $\mathcal{A}_{\infty,s}$ in $s$. However, to prove modified scattering it is enough to take a fixed $s$, which makes statements about the domain of convergence cleaner. 
\end{remark}
\begin{proof}
The proof involves several steps, with the key tools being the convergence of the velocity characteristics (Proposition \ref{Vlim}) and the asymptotics for $E$ along the trajectories of free transport (Proposition \ref{Eerrest}). The estimates on the characteristics from Lemmas \ref{forflowdbds} and \ref{Ybds} will also be used extensively. The convergence of derivatives of the spatial characteristics will require a somewhat delicate ``freezing" argument in the long-range correction term to avoid derivative loss. Then, the convergence estimates for derivatives will imply that $Z_\infty(s)$ is $C^1$, volume preserving, and injective for every $s$.\\

\textbf{a. Uniform convergence.} For the first claim of the proposition, we compute 
    $$\frac{d}{d\tau}Z(\tau,s,x,v)=-\tau E\big(Y(\tau)+\tau V(\tau)\big) + \tau^{-1}E_\infty(V(\tau)) + \ln(\tau) \nabla E_\infty(V(\tau)) E(\tau,X(\tau)).$$
    Then, using Proposition \ref{Eerrest}, Lemma \ref{Ybds}, and the decay of $E$, we obtain
    \begin{align*}
        \big|\tfrac{d}{d\tau}Z(\tau,s,x,v)\big| &\les \tau \big( \eps_0^2 \jbrac{Y(\tau)}^{2\delta}\tau^{-2-\delta} \big) + \ln(\tau) \|\nabla E_\infty\|_{L^{\infty}} \eps_0^2 \jbrac{\tau}^{-2}\\
        &\les \begin{cases}
            \eps_0^2 \tau^{-1-\delta}\ln^{2\delta}\jbrac{\tau} & (x,v)\in U_s,\\
            \eps_0^2 \tau^{-1-\delta}(\jbrac{s}^{2\delta}+\ln^{2\delta}\jbrac{\tau}) & (x,v)\in U_0.
        \end{cases}
    \end{align*}
    Again we may use the fact that the condition on $\delta$ is open in order to absorb the logarithmic factor in $\tau$. Using the fundamental theorem of calculus, we get the desired convergence of $Z(s)$ to $Z_\infty$:
    \begin{equation}\label{C0Zconv}
        |Z(t,s,x,v)-Z_\infty(s,x,v)| = \left| \int_t^\infty \frac{d}{d\tau}Z(\tau,s,x,v)d\tau \right| \les \begin{cases}
            \eps_0^2 t^{-\delta} & (x,v)\in U_s,\\
            \eps_0^2 t^{-\delta}\jbrac{s}^{2\delta} & (x,v)\in U_0.
        \end{cases}
    \end{equation}
    Moreover, the same estimate applied to the definition of $Z_\infty$ yields
    $$Z_\infty(s,x,v)=x-vs+\ln(s)E_\infty(v)+O(s^{\delta}).$$

    \textbf{b. Convergence of derivatives.} In order to prove that the asymptotic flow is $C^1$ we need to prove convergence estimates for derivatives of $Z$. If we tried to repeat the above argument for derivatives, we would apparently require $E_\infty\in W^{2,\infty}$, but we only have $E_\infty\in C^{1,\delta}$ so we need a different approach. To avoid putting another derivative on $E_\infty$, we remove the $t$-dependence in the argument of $E_\infty$ by ``freezing" the velocity at infinity, i.e. replacing the velocity characteristic by the asymptotic velocity:
    $$\tilde{Z}(t,s,x,v) := Y(t,s,x,v) + \ln(t)E_\infty(V_\infty(s,x,v)).$$
    First, we show that $\tilde{Z}$ converges in $W^{1,\infty}$ to a limit $\tilde{Z}_\infty$, defined in the natural way as
    $$\tilde{Z}_\infty(s,x,v) := x-vs + \ln(s)E_\infty(V_\infty(s,x,v)) + \int_s^\infty \frac{d}{d\tau} \tilde{Z}(\tau,s,x,v)d\tau.$$
    In a nearly identical manner as the estimate on $\tfrac{d}{d\tau}Z(\tau,s,x,v)$, Propositions \ref{Eerrest} and \ref{Vlim} and Lemma \ref{Ybds} give
    \begin{align*}
        \big|\tfrac{d}{d\tau}\tilde{Z}(\tau,s,x,v)\big|&=|-\tau E(\tau,Y(\tau)+\tau V(\tau))+\tau^{-1}E_\infty(V_\infty)|\\
        &\leq \tau|E(\tau,Y(\tau)+\tau V(\tau))-\tau^{-2}E_\infty(V(\tau))|+\tau^{-1}|E_\infty(V(\tau))-E_\infty(V_\infty)|\\
        &\les \eps_0^2 \tau^{-1-\delta}\jbrac{Y(\tau)}^{2\delta}  + \tau^{-1}[E_\infty]_{C^\delta}|V(\tau)-V_\infty|^{\delta}\\
        &\les \begin{cases}
            \eps_0^2 \tau^{-1-\delta}\ln^{2\delta}\jbrac{\tau} & (x,v)\in U_s,\\
            \eps_0^2 \tau^{-1-\delta}(\jbrac{s}^{2\delta}+\ln^{2\delta}\jbrac{\tau}) & (x,v)\in U_0.
        \end{cases}
    \end{align*} 
    Just as before, we absorb the logarithmic factor and use the fundamental theorem of calculus, implying that 
    $$|\tilde{Z}(t,s,x,v)-\tilde{Z}_\infty(s,x,v)| \les \begin{cases}
            \eps_0^2 t^{-\delta} & (x,v)\in U_s,\\
            \eps_0^2 t^{-\delta}\jbrac{s}^{2\delta} & (x,v)\in U_0.
        \end{cases}$$
    Now we turn to convergence of the derivatives of $\tilde{Z}$. To simplify the notation we write $V_\infty = V_\infty(s,x,v)$ and $V = V(\tau,s,x,v)$ in the following estimates. With $z$ standing in for either $x$ or $v$,
    \begin{align*}
        \big|\tfrac{d}{d\tau}\nabla_z\tilde{Z}(\tau,s,x,v)\big| &= \big|-\tau (\nabla_x E)(\tau,Y+\tau V)\nabla_z( Y+\tau V) + \tau^{-1}\nabla E_\infty(V_\infty)\nabla_z V_\infty \big|\\
        &\leq \tau^2|(\nabla_x E)(\tau,Y+\tau V) - \tau^{-3}\nabla E_\infty(V)||\nabla_z V| \\
        &\quad+ \tau^{-1}|\nabla E_\infty(V) - \nabla E_\infty(V_\infty)||\nabla_z V|+ \tau^{-1} |\nabla E_\infty(V_\infty)||\nabla_z(V-V_\infty)| \\
        &\quad+ \tau|(\nabla_x E)(\tau,Y+\tau V)||\nabla_z Y|.
    \end{align*}
    
    Using the same estimates used to prove the convergence of $Z,\tilde{Z}$ (namely, Propositions \ref{Eerrest} and \ref{Vlim}) and in addition using the decay of $\nabla_x E$ and Lemmas \ref{forflowdbds} and \ref{Ybds}, we obtain
    \begin{align*}
        \big|\tfrac{d}{d\tau}\nabla_x\tilde{Z}(\tau,s,x,v)\big| &\les \eps_0^2\tau^{-1-\delta}\jbrac{Y}^{2\delta} + \eps_0^2\tau^{-1}[\nabla E_\infty]_{C^\delta}|V_\infty-V|^\delta + \eps_0^2\tau^{-3}\ln\jbrac\tau + \eps_0^2\tau^{-2}\ln^3\jbrac{\tau}\\
        &\les\begin{cases}
            \eps_0^2 \tau^{-1-\delta}\ln^{2\delta}\jbrac{\tau} & (x,v)\in U_s,\\
            \eps_0^2 \tau^{-1-\delta}(\jbrac{s}^{2\delta}+\ln^{2\delta}\jbrac{\tau}) & (x,v)\in U_0,
        \end{cases}\\
        \big|\tfrac{d}{d\tau}\nabla_v\tilde{Z}(\tau,s,x,v)\big| &\les \eps_0^2\tau^{-1-\delta}\jbrac{Y}^{2\delta} + \eps_0^2\tau^{-1}[\nabla E_\infty]_{C^\delta}|V_\infty-V|^\delta + \eps_0^2\tau^{-3}\ln\jbrac\tau + \eps_0^2\tau^{-2}\ln^3\jbrac{\tau}(s+\eps_0^2\ln^2\jbrac{\tau})\\
        &\les\eps_0^2 \tau^{-1-\delta}\jbrac{s}\ln^{2\delta}\jbrac{\tau}, \quad (x,v)\in U_0\cup U_s.
    \end{align*}
    Therefore, again absorbing the logarithmic factors into the open condition on $\delta$,
    \begin{align*}
        |\nabla_x\tilde{Z}_\infty(s,x,v) - \nabla_x \tilde{Z}(t,s,x,v)| &\les \begin{cases}
            \eps_0^2 t^{-\delta} & (x,v)\in U_s,\\
            \eps_0^2 t^{-\delta}\jbrac{s}^{2\delta} & (x,v)\in U_0,
        \end{cases}\\
        |\nabla_v\tilde{Z}_\infty(s,x,v) - \nabla_v \tilde{Z}(t,s,x,v)| &\les \eps_0^2 t^{-\delta}\jbrac{s}, \quad (x,v)\in U_0\cup U_s
    \end{align*}    
    
    Next, we will show that $\tilde{Z}(t)$ gets asymptotically close to $Z(t)$ as $t\to\infty$ in $W^{1,\infty}_{x,v}$. First, using Proposition \ref{Vlim} we get
    \begin{align*}
        |\tilde{Z}(t,s,x,v)-Z(t,s,x,v)| \leq \ln(t)\|\nabla E_\infty\|_{L^\infty}|V(t)-V_\infty| \les \eps_0^2 \jbrac{t}^{-1}\ln(t).
    \end{align*}
    Therefore, since $Z(t)\to Z_\infty$ and $\tilde{Z}(t)\to \tilde{Z}_\infty$ in $L^\infty$, we must have $Z_\infty = \tilde{Z}_\infty$, since
    $$|Z_\infty-\tilde{Z}_\infty|\leq |Z_\infty-Z(t)|+|Z(t)-\tilde{Z}(t)|+|\tilde{Z}(t)-\tilde{Z}_\infty| \to 0 \quad \text{ as }t\to \infty.$$
    Crucially, this means that once we prove that $\tilde{Z}(t)$ and $Z(t)$ are close in $W^{1,\infty}_{x,v}$, we will get the $W^{1,\infty}_{x,v}$ convergence of $Z(t)$ from the $W^{1,\infty}_{x,v}$ convergence of $\tilde{Z}(t)$. To this end, using Lemma \ref{forflowdbds} and Proposition \ref{Vlim},
    \begin{align*}
        |\nabla_x \tilde{Z}(t,s,x,v)-\nabla_x Z(t,s,x,v)| &\leq \ln(t)\big(|\nabla E_\infty(V)-\nabla E_\infty(V_\infty)||\nabla_x V| + |\nabla E_\infty(V_\infty)||\nabla_x (V-V_\infty)|\big)\\
        &\les \ln(t)\big( \eps_0[\nabla E_\infty]_{C^\delta}|V-V_\infty|^\delta + \eps_0^2\jbrac{t}^{-2}\ln\jbrac{t} \big)\\
        &\les \eps_0^2\jbrac{t}^{-\delta}\ln(t).
    \end{align*}
    A nearly identical argument for $v$ derivatives gives
    \begin{align*}
        |\nabla_v \tilde{Z}(t,s,x,v)-\nabla_v Z(t,s,x,v)| \les \eps_0^2 \jbrac{t}^{-\delta}\ln(t).
    \end{align*}
    Using the fact that $Z_\infty=\tilde{Z}_\infty$ and the convergence estimates just proved,  
    \begin{align*}
        |\nabla_x Z(t,s,x,v)-\nabla_x Z_\infty(s,x,v)| &\leq |\nabla_x Z(t) - \nabla_x \tilde{Z}(t)| + |\nabla_x \tilde{Z}(t)-\nabla_x \tilde{Z}_\infty|\\
        &\les \eps_0^2\jbrac{t}^{-\delta}\ln(t) + \begin{cases}
            \eps_0^2 t^{-\delta} & (x,v)\in U_s,\\
            \eps_0^2 t^{-\delta}\jbrac{s}^{2\delta} & (x,v)\in U_0.
        \end{cases}
    \end{align*}
    Similarly,
    \begin{align*}
        |\nabla_v Z(t,s,x,v)-\nabla_v Z_\infty(s,x,v)| &\leq |\nabla_v Z(t) - \nabla_v \tilde{Z}(t)| + |\nabla_v \tilde{Z}(t)-\nabla_v \tilde{Z}_\infty|\\
        &\les \eps_0\jbrac{t}^{-\delta}\ln(t) + \eps_0 t^{-\delta}\jbrac{s}.
    \end{align*}
    
    \textbf{c. Volume-preserving property and injectivity.} Collecting estimates and absorbing logarithmic factors into the open condition on $\delta$, we see that 
    $$\|\mathcal{A}_{t,s}-\mathcal{A}_{\infty,s}\|_{W^{1,\infty}_{x,v}(U_0)}\les \eps_0^2 \jbrac{s} t^{-\delta}.$$
    Since $\mathcal{A}_{t,s}$ is volume preserving (it is the composition of two volume preserving maps; the nonlinear flow and a shear flow), the $C^1$ convergence implies that $\mathcal{A}_{\infty,s}$ is volume preserving for every fixed $s$:
    $$\det \nabla_{x,v} \mathcal{A}_{\infty,s} = \lim_{t\to\infty}\det \nabla_{x,v} \mathcal{A}_{t,s} = 1, \quad \text{ on } U_0.$$
    
    From this, the inverse function theorem gives local injectivity of $\mathcal{A}_{\infty,s}$, but to obtain global injectivity, we need the following topological lemma. It appears standard but we include a proof since we were unable to find it in the literature.
    \begin{lem}\label{toplem}
        Suppose $U$ is an open set and $f_n:U\to \R^d$ are injective continuous functions converging uniformly on $U$ to a locally injective function $f$. Then $f$ is injective.
    \end{lem}
    \begin{proof}
        Suppose for the sake of contradiction that there are $x_1\neq x_2$ with $f(x_1)=f(x_2):=y$. Choose $\eps>0$ small enough so that $B_i := B(x_i,\eps)$ are disjoint, contained in $U$, and $f|_{B_i}$ are injective. By invariance of domain, $f|_{B_i}$ are open maps, so there exists an open ball $V$ centered at $y$ whose closure is contained in $f|_{B_1}(B_1)\cap f|_{B_2}(B_2)$. For $x\in \overline{V}$ and $i\in\{1,2\}$ define $\varphi_{n,i}(x) = y+x-f_n(f|_{B_i}^{-1}(x))$. Since $f_n\to f$ uniformly on $U$, $\varphi_{n,i}\to y$ uniformly on $\overline{V}$, so by choosing $n$ large, $\varphi_{n,i}$ is a map from $\overline{V}$ to itself and hence has a fixed point $p_i$ by Brouwer's theorem. This implies $f_n(f|_{B_i}^{-1}(p_i)) = y$ for $i\in\{1,2\}$, which contradicts the injectivity of $f_n$ since the $B_i$ are disjoint.
    \end{proof}

    Since $\mathcal{A}_{t,s}$ is injective and converges uniformly to $\mathcal{A}_{\infty,s}$ on $U_0$ as $t\to\infty$ by (\ref{C0Zconv}) and Proposition \ref{Vlim}, the Lemma yields the global injectivity of $Z_\infty$. The inverse function theorem also shows that the inverse is $C^1$, concluding the proof that $\mathcal{A}_{\infty,s}$ is a diffeomorphism from $U_0$ onto its image.
\end{proof}

\subsection{Asymptotic Vlasov profile}

In this section, we show the existence of a profile $f_\infty$ such that $f$ converges in $L^p$ along the modified characteristics to $f_\infty$. This will prove (\ref{fasymp}). Recall that we write $\mathcal{S}_t$ for the support of $f(t)$: 
$$\mathcal{S}_t := \{(x,v)\in \R^6: f(t,x,v)\neq 0\}.$$

\begin{prop}\label{profileconvprop}
    There exists $f_\infty\in L^\infty$ with compact support such that 
    $$\lim_{t\to\infty}\|f(t)\circ \mathcal{M}_t - f_\infty\|_{L^p_{x,v}}=0, \quad 1\leq p<\infty.$$
\end{prop}
\begin{proof}
The proof will proceed in a few steps. First we will show weak convergence against continuous compactly supported functions (i.e. weak convergence in the sense of measures). Then, via a density argument and uniform boundedness along the flow, we will upgrade this to weak $L^2$ convergence, which, using the fact that $L^2$ norms remain invariant under the relevant flows, actually implies strong $L^2$ convergence, from which we may obtain $L^p$ convergence for any finite $p$ by interpolation.\\

Let $\varphi\in C_c(\R^3\times\R^3)$ be a continuous compactly supported function. Then using the invertibility and volume preserving property of $\Phi_{t,s}$ and $\mathcal{M}_t$ (and recalling that $\Phi_{t,s}^{-1}=\Phi_{s,t}$), we get
\begin{align*}
    \jbrac{ \varphi, f(t)\circ \mathcal{M}_t } &= \jbrac{ \varphi\circ\mathcal{M}_t^{-1}, f(t) }\\
    &= \jbrac{ \varphi\circ\mathcal{M}_t^{-1}, f(s)\circ\Phi_{s,t} }\\
    &= \jbrac{ \varphi\circ\mathcal{M}_t^{-1}\circ\Phi_{t,s}, f(s) }\\
    &= \jbrac{ \varphi\circ \mathcal{A}_{t,s},f(s)}.
\end{align*}
Therefore, Propositions \ref{Vlim} and \ref{Xlim}, the continuity of $\varphi$, and the invertibility and volume preserving property of $\mathcal{A}_{\infty,s}$ yield
\begin{equation}\label{profileconv}
    \lim_{t\to\infty}\jbrac{ \varphi, f(t)\circ \mathcal{M}_t } = \jbrac{ \varphi\circ \mathcal{A}_{\infty,s},f(s) \chi_{\mathcal{S}_s}} = \jbrac{ \varphi,f(s) \circ \mathcal{A}_{\infty,s}^{-1} \chi_{\mathcal{A}_{\infty,s}(\mathcal{S}_s)}}.
\end{equation}
We therefore define 
$$f_\infty := (f(s)\circ \mathcal{A}_{\infty,s}^{-1})\chi_{\mathcal{A}_{\infty,s}(\mathcal{S}_s)},$$
which is clearly bounded and compactly supported as $\mathcal{A}_{\infty,s}$ is continuous and $\mathcal{S}_s$ is bounded. As a consequence of the volume-preserving property of $\mathcal{A}_{\infty,s}$, $\mathcal{M}_t$, and $\Phi_{t,s}$, we have 
\begin{equation}\label{L2pres}
    \|f_\infty\|_{L^2_{x,v}} = \|f(s)\|_{L^2_{x,v}} = \|f_0\|_{L^2_{x,v}} = \|f(t)\circ \mathcal{M}_t\|_{L^2_{x,v}}.
\end{equation}
This is enough to obtain weak $L^2$ convergence: given $\eta>0$ and $h\in L^2(\R^6)$, by density we can find $\varphi\in C_c(\R^6)$ with $\|\varphi-h\|_{L^2_{x,v}}\leq \eta$, so that
\begin{align*}
    |\jbrac{h,f(t)\circ\mathcal{M}_t-f_\infty}| \leq |\jbrac{\varphi,f(t)\circ\mathcal{M}_t-f_\infty}| + 2\|f_0\|_{L^2_{x,v}}\|\varphi-h\|_{L^2_{x,v}}.
\end{align*}
The first term converges to zero in view of (\ref{profileconv}), and the second term can be made as small as we wish by choosing $\eta$ small. This yields weak $L^2$ convergence, but using (\ref{L2pres}), we can upgrade this to strong $L^2$ convergence:
\begin{align*}
    \lim_{t\to\infty}\|f(t)\circ \mathcal{M}_t - f_\infty\|_{L^2}^2 &= \lim_{t\to\infty}\|f(t)\circ \mathcal{M}_t\|_{L^2}^2 + \|f_\infty\|_{L^2}^2 - 2\lim_{t\to\infty}\jbrac{f(t)\circ \mathcal{M}_t, f_\infty}\\
    &= \|f_\infty\|_{L^2}^2 + \|f_\infty\|_{L^2}^2 - 2\|f_\infty\|_{L^2}^2 = 0.
\end{align*}
Now we deal with $L^p$ convergence. Observe that $f(t)\circ \mathcal{M}_t$ is zero outside of $\mathcal{A}_{t,s}(\mathcal{S}_s)$ since $$\mathcal{A}_{t,s}(\mathcal{S}_s)=\mathcal{M}_t^{-1}(\Phi_{t,s}(\mathcal{S}_s)) = \mathcal{M}_t^{-1}(\mathcal{S}_t).$$
Therefore, since $\mathcal{A}_{t,s}\to\mathcal{A}_{\infty,s}$ uniformly on $\mathcal{S}_s$, and $f_\infty$ is supported on $\mathcal{A}_{\infty,s}(\mathcal{S}_s)$, for large enough $t$ the supports of $f(t)\circ \mathcal{M}_t$ and $f_\infty$ are both contained in a common compact set independent of $t$. Therefore, H\"older yields
$$\|f(t)\circ \mathcal{M}_t - f_\infty\|_{L^1}\les \|f(t)\circ \mathcal{M}_t - f_\infty\|_{L^2}\to 0.$$
Interpolation then gives $L^p$ convergence for $1\leq p\leq 2$. Similarly, since 
$$\|f(t)\circ \mathcal{M}_t - f_\infty\|_{L^\infty}\leq 2\|f_0\|_{L^\infty},$$
interpolation also gives $L^p$ convergence for $2\leq p<\infty$.

\end{proof}

\printbibliography

\end{document}